\documentclass[12pt, reqno]{amsart}
\usepackage{}
\usepackage{stmaryrd}
\usepackage{amsmath}
\usepackage{amsfonts}
\usepackage{amssymb}
\usepackage[all,cmtip]{xy}           %xypic macro for latex2.09
\usepackage{xspace}
\usepackage{bbding}
\usepackage{txfonts}
\usepackage[shortlabels]{enumitem}
\usepackage{cite}
\usepackage{tikz}
\usepackage{titletoc}

\usepackage{ifpdf}
\ifpdf
 \usepackage[colorlinks,final,backref=page,hyperindex]{hyperref}
\else
 \usepackage[colorlinks,final,backref=page,hyperindex,hypertex]{hyperref}
\fi
\usepackage[active]{srcltx}

\makeatletter

\begin{document}
\newcommand {\emptycomment}[1]{} %to remove paragraphs

\baselineskip=15pt
\newcommand{\nc}{\newcommand}
\newcommand{\delete}[1]{}
\nc{\mfootnote}[1]{\footnote{#1}} % Use this to show footnotes
\nc{\todo}[1]{\tred{To do:} #1}

%\delete{
\nc{\mlabel}[1]{\label{#1}}  % Use this to suppress names
\nc{\mcite}[1]{\cite{#1}}  % Use this to suppress names
\nc{\mref}[1]{\ref{#1}}  % Use this to suppress names
\nc{\meqref}[1]{\eqref{#1}} % Use this to suppress names
\nc{\mbibitem}[1]{\bibitem{#1}} % Use this to show number
%}

\delete{
\nc{\mlabel}[1]{\label{#1}  % Use the next two lines to show names
{\hfill \hspace{1cm}{\bf{{\ }\hfill(#1)}}}}
\nc{\mcite}[1]{\cite{#1}{{\bf{{\ }(#1)}}}}  % Use this lines to show names
\nc{\mref}[1]{\ref{#1}{{\bf{{\ }(#1)}}}}  % Use this lines to show names
\nc{\meqref}[1]{\eqref{#1}{{\bf{{\ }(#1)}}}} % Use this lines to show names
\nc{\mbibitem}[1]{\bibitem[\bf #1]{#1}} % Use this to show name
}

\newcommand {\comment}[1]{{\marginpar{*}\scriptsize\textbf{Comments:} #1}}
\nc{\mrm}[1]{{\rm #1}}
\nc{\id}{\mrm{id}}  \nc{\Id}{\mrm{Id}}
%%%%%%%%%%%%%%%%%%%%%%%%

\def\a{\alpha}
\def\b{\beta}
\def\bd{\boxdot}
\def\bbf{\bar{f}}
\def\bF{\bar{F}}
\def\bbF{\bar{\bar{F}}}
\def\bbbf{\bar{\bar{f}}}
\def\bg{\bar{g}}
\def\bG{\bar{G}}
\def\bbG{\bar{\bar{G}}}
\def\bbg{\bar{\bar{g}}}
\def\bT{\bar{T}}
\def\bt{\bar{t}}
\def\bbT{\bar{\bar{T}}}
\def\bbt{\bar{\bar{t}}}
\def\bR{\bar{R}}
\def\br{\bar{r}}
\def\bbR{\bar{\bar{R}}}
\def\bbr{\bar{\bar{r}}}
\def\bu{\bar{u}}
\def\bU{\bar{U}}
\def\bbU{\bar{\bar{U}}}
\def\bbu{\bar{\bar{u}}}
\def\bw{\bar{w}}
\def\bW{\bar{W}}
\def\bbW{\bar{\bar{W}}}
\def\bbw{\bar{\bar{w}}}
\def\btl{\blacktriangleright}
\def\btr{\blacktriangleleft}
\def\c{\cdot}
\def\ci{\circ}
\def\d{\delta}
\def\dd{\diamondsuit}
\def\D{\Delta}
\def\G{\Gamma}
\def\g{\gamma}
\def\k{\kappa}
\def\l{\lambda}
\def\lr{\longrightarrow}
\def\o{\otimes}
\def\om{\omega}
\def\p{\psi}
\def\r{\rho}
\def\ra{\rightarrow}
\def\rh{\rightharpoonup}
\def\lh{\leftharpoonup}
\def\s{\sigma}
\def\st{\star}
\def\ti{\times}
\def\tl{\triangleright}
\def\tr{\triangleleft}
\def\v{\varepsilon}
\def\vp{\varphi}

%%%%%%%%%%%%%%%%%%%%%%%% Statements
\newtheorem{thm}{Theorem}[section]
\newtheorem{lem}[thm]{Lemma}
\newtheorem{cor}[thm]{Corollary}
\newtheorem{pro}[thm]{Proposition}
\theoremstyle{definition}
\newtheorem{defi}[thm]{Definition}
\newtheorem{ex}[thm]{Example}
\newtheorem{rmk}[thm]{Remark}
\newtheorem{pdef}[thm]{Proposition-Definition}
\newtheorem{condition}[thm]{Condition}
\newtheorem{question}[thm]{Question}
\renewcommand{\labelenumi}{{\rm(\alph{enumi})}}
\renewcommand{\theenumi}{\alph{enumi}}

\nc{\ts}[1]{\textcolor{purple}{Tianshui:#1}}
\nc{\jie}[1]{\textcolor{blue}{LIJIE:#1}}
\font\cyr=wncyr10

%%%%%%%%%%%%%%%%%%%%%%%%%%%%%%%%%%%%%%%%%%%%%%%%%%%%%%%%%%%%%%%%%%%%%%%%%%%%%%%%%%%%%%%%%%%%%%%%%%%%%%%%%%%%%%%%%%%%%%%%%%%%%%%%%%%%
 \title[Nonhomogeneous aYBes]{\bf Nonhomogeneous associative Yang-Baxter equations}

 \author[T. Ma]{Tianshui Ma\textsuperscript{*}}
 \address{School of Mathematics and Information Science, Henan Normal University, Xinxiang 453007, China}
         \email{matianshui@htu.edu.cn}

 \author[J. Li]{Jie Li}
 \address{School of Mathematics and Information Science, Henan Normal University, Xinxiang 453007, China}
         \email{{lijie\_0224@163.com}}

  \thanks{\textsuperscript{*}Corresponding author}

\date{\today}

 \begin{abstract}
  We introduce the notion of associative (BiHom-)Yang-Baxter pair of weight $(\lambda,\gamma)$ which can provide the solution to the double curved Rota-Baxter (BiHom-)system. Equivalent characterizations of (quasitriangular) covariant BiHom-bialgebra are given. We also prove that associative BiHom-Yang-Baxter equation of weight $-1$ can be obtained by the unitary quasitriangular covariant BiHom-bialgebra. At last, we present two approaches to construct (BiHom-)pre-Lie modules from Rota-Baxter (BiHom-)paired modules.
 \end{abstract}

 \keywords{Nonhomogeneous associative Yang-Baxter equations, Rota-Baxter system, Covariant bialgebra}

\subjclass[2020]{
17B38,%17B38 Yang-Baxter equations and Rota-Baxter operators
17A30,%17A30 Nonassociative algebras satisfying other identities
%17B61,%17B61 Hom-Lie and related algebras
17D30,%17D30 (non-Lie) Hom algebras and topics
%16T05,%16T05 Hopf algebras and their applications [See also 16S40, 57T05]
16T10.%16T10 Bialgebras
%16T25,%16T25 Yang-Baxter equations
%17B62 %17B62 Lie bialgebras; Lie coalgebras
%17B63 %17B63 Poisson algebras
}

 \maketitle

\tableofcontents

\numberwithin{equation}{section}
\allowdisplaybreaks

%\newtheorem{definition}{Definition}[section]
%\newtheorem{theorem}{Theorem}[section]
%\newtheorem{lemma}{Lemma}[section]
%\newtheorem{remark}{Remark}[section]
%\newtheorem{proposition}{Proposition}[section]
%\newtheorem{corollary}{Corollary}[section]

%@=h长度i 设置元素间距
%@R=h长度i 设置行间距
%@C=h长度i 设置列间距
%@M=h长度i 设置元素的默认边距
%@W=h长度i 设置元素的默认宽度
%@H=h长度i 设置元素的默认高度
%@L=h长度i 设置标签的边距
%@C=3em@R=8ex@M=1em@W=0ex@H=1em@L=1ex

 \section{Introduction and preliminaries}\mlabel{se:int} The associative Yang-Baxter equation (aYBe) \cite{A01} is analogous to the classical Yang-Baxter equation (cYBe) which is named after C. N. Yang and R. J. Baxter. Rota-Baxter operators \cite{Ba60}, on the other hand, are named after G. Baxter, which is widely used in many fields, such as pure and applied mathematics, and more recently in mathematical physics. One can refer to Guo's book \cite{Guo1} for the detailed theory of Rota-Baxter algebras. The relation between Rota-Baxter operator on Lie algebra and cYBe is closely. Rota-Baxter operator of weight 0 can be regarded as an operator form of the cYBe. To study the representation of Rota-Baxter algebra, Zheng, Guo, Zhang introduced the notion of Rota-Baxter paired module and investigate the properties.

 Recently, in \cite{Br}, Brzezi\'{n}ski introduced the notions of associative Yang-Baxter pair (aYBp) and Rota-Baxter system (the curved version was given in \cite{Br2}), then constructed a class of bialgebra (named covariant bialgebra) via two derivations, which extends Joni, Rota's infinitesimal bialgebra \cite{JR}, and found that some quasitriangular covariant bialgebras give rise to Rota-Baxter systems. In \cite{GMMP}, Graziani, Makhlouf, Menini and Panaite introduced algebras in the group Hom-category, which are called BiHom-associative algebras and involve two commuting multiplicative linear maps. In \cite{LMMP3}, Liu, Makhlouf, Menini and Panaite introduced and studied (quasitriangular) Joni, Rota's infinitesimal BiHom-bialgebras. In \cite{MLi}, Ma and Li proposed the notion of nonhomogeneous associative BiHom-Yang-Baxter equations based on the structure of coboundary $\l$-infinitesimal BiHom-bialgebra (the BiHom-version of Ebrahimi-Fard's weighted infinitesimal bialgebra including Joni, Rota's infinitesimal bialgebras \cite{JR} and Loday, Ronco's infinitesimal bialgebras  \cite{LR}) which covers Liu, Makhlouf, Menini and Panaite's version in \cite{LMMP3} for the BiHom-type of cYBe when the weight $\l=0$ and Ebrahimi-Fard's version in \cite{E06} for the nonhomogeneous type of cYBe when the structure maps $\a=\b=\psi=\om=\id$. More research on BiHom-type algebras can be found in \cite{CMM,CMM20,LC,LMMP,LMMP1,LMMP2,LMMP3,LMMP6,MLiY,MY,MYZZ,ZW}, etc.

 Based on the above research, it is natural to consider that what nonhomogeneous associative (BiHom-)Yang-Baxter pair is and what is the relationship with other algebraic structures, such as curved Rota-Baxter (BiHom-)system, curved weak BiHom-pseudotwistor, covariant BiHom-bialgebra, etc. In this paper, we investigate the relations among them in the setting of BiHom-case. As the application of Rota-Baxter (BiHom-)system, we provide two approaches to construct (BiHom-)pre-Lie modules from Rota-Baxter (BiHom-)paired modules.

 The layout of the paper is as follows. In Section \mref{se:yb}, we present the notions of nonhomogeneous associative (BiHom-)Yang-Baxter pair and double curved Rota-Baxter (BiHom-)system which can cover the homogeneous (BiHom-)associative Yang-Baxter pair in \cite{Br}, the BiHom-Yang-Baxter equation in \cite{LMMP3} and (curved) Rota-Baxter system in \cite{Br,Br2}. The relations among nonhomogeneous associative (BiHom-)Yang-Baxter pair, double curved Rota-Baxter (BiHom-)system, double curved weak BiHom-pseudotwistor and BiHom-associative algebra are discussed (see Theorems \ref{thm:4.46}, \ref{thm:8.1a} and \ref{thm:8.4}). Equivalent characterizations of (quasitriangular) covariant BiHom-bialgebra are also given (see Theorem \mref{thm:2.9}), which can recover \cite[Theorem 3.5]{MY} and \cite[Proposition 3.10]{Br}. And also we prove that the solution to nonhomogeneous associative BiHom-Yang-Baxter equation (abhYBe) can be provided by unitary quasitriangular covariant BiHom-bialgebra (Proposition \ref{pro:2.17}). As a special case, we really obtain the BiHom-version of quasitriangular infinitesimal bialgebra by four structure maps (Proposition \ref{pro:2.11}) which is different from the BiHom-version in \cite{LMMP3}.

 Throughout this paper, $K$ will be a field, and all vector spaces, tensor products, and
 homomorphisms are over $K$. Next we recall some useful definitions which will be used later. Now let us recall some basic notions from \cite{GMMP} in the theory of BiHom-type algebra.

 \begin{defi}\mlabel{de:1.1} A {\bf BiHom-associative algebra} is a 4-tuple $(A,\mu,\alpha,\beta)$, where $A$ is a linear space, $\alpha,\beta:A\lr A$ and $\mu:A\otimes A\lr A$ (write $\mu(a\otimes b)=ab$) are linear maps satisfying the following conditions: $\alpha\circ\beta=\beta\circ\alpha$ and for all $a,b,c\in A$,
 \begin{eqnarray}
 &\alpha(ab)=\alpha(a)\alpha(b),\quad\beta(ab)=\beta(a)\beta(b),&\mlabel{eq:1.2}\\
 &\alpha(a)(bc)=(ab)\beta(c).&\mlabel{eq:1.3}
 \end{eqnarray}
 \end{defi}
 A BiHom-associative algebra $(A,\mu,\alpha,\beta)$ is called {\bf unitary} if there exists an element $1_{A}\in A$ (called a unit) such that
 \begin{eqnarray}
 &\alpha(1_{A})=1_{A},\quad\beta(1_{A})=1_{A},\quad a1_{A}=\alpha(a),\quad 1_{A}a=\beta(a),\quad \forall a\in A.&\mlabel{eq:1.5}
 \end{eqnarray}

 \begin{defi}\mlabel{de:1.2} A {\bf BiHom-coassociative coalgebra} is a 4-tuple $(C,\Delta,\psi,\omega)$, in which $C$ is a linear space, $\psi,\omega:C\lr C$ and $\Delta:C\lr C\otimes C$ are linear maps, such that $\psi\circ\omega=\omega\circ \psi$ and
 \begin{eqnarray}
 &(\psi\otimes\psi)\circ\Delta=\Delta\circ\psi,~~
 (\omega\otimes\omega)\circ\Delta=\Delta\circ\omega,&\mlabel{eq:1.7}\\
 &(\Delta\otimes\psi)\circ\Delta=(\omega\otimes \Delta)\circ\Delta.&\mlabel{eq:1.9}
 \end{eqnarray}
 \end{defi}

 A BiHom-coassociative coalgebra $(C,\Delta,\psi,\omega)$ is called {\bf counitary} if there exists a linear map $\varepsilon:C\lr K $ (called a counit) such that
 \begin{eqnarray}
 &\varepsilon\circ\psi=\varepsilon,\quad \varepsilon\circ\omega=\varepsilon,\quad (\id_{C}\otimes\varepsilon)\circ\Delta=\omega,\quad(\varepsilon\otimes \id_{C})\circ\Delta=\psi.&\mlabel{eq:1.11}
 \end{eqnarray}

 \begin{defi}\mlabel{de:1.3} Let $(A,\mu,\alpha_{A},\b_{A})$ be a BiHom-associative algebra. A {\bf left $(A,\mu,\alpha_{A},\beta_{A})$-module} is a 4-tuple $(M,\triangleright,\alpha_{M},\beta_{M})$,  where $M$ is a linear space,  $\alpha_{M}, \beta_{M}: M\lr M$ and $\triangleright: A\otimes M\lr M$ are linear maps such that, $\alpha_{M}\circ\beta_{M}=\beta_{M}\circ\alpha_{M}$, and for all $a,b\in A, m\in M$,
 \begin{eqnarray}
 &\alpha_{M}(a\triangleright m)=\alpha_{A}(a)\triangleright\alpha_{M}(m),~~\beta_{M}(a\triangleright m)=\beta_{A}(a)\triangleright\beta_{M}(m),&\mlabel{eq:1.13}\\
 &\alpha_{A}(a)\triangleright(b\triangleright m)=(ab)\triangleright\beta_{M}(m).&\mlabel{eq:1.15}
 \end{eqnarray}
 Likewise, we can get the right version of $(A,\mu,\alpha_{A},\beta_{A})$-module.

 If $(M,\triangleright,\alpha_{M},\beta_{M})$ is a left $(A,\mu,\alpha_{A},\beta_{A})$-module and at the same time $(M,\triangleleft,\alpha_{M},\beta_{M})$ is right $(A,\mu,\alpha_{A},\beta_{A})$-module, then $(M,\triangleright,\triangleleft,\alpha_{M},\beta_{M})$ is an {\bf $(A,\mu,\alpha_{A},\beta_{A})$-bimodule} if
 \begin{eqnarray}
 & \alpha_{A}(a)\triangleright(m\triangleleft b)=(a\triangleright m)\triangleleft \beta_{A}(b).&\mlabel{eq:1.16}
 \end{eqnarray}
 \end{defi}

 \begin{rmk}
 When the structure maps $\a=\b=\id$ in Definitions \mref{de:1.1}-\mref{de:1.3}, then we can get the usual definitions of associative algebra, coassociative coalgebra, bimodule, respectively.
 \end{rmk}

 \section{Nonhomogeneous abhYBp and double curved Rota-Baxter (BiHom-)system}\label{se:yb} In this section, we aim to investigate the nonhomogeneous type and meanwhile BiHom-type for the associative Yang-Baxter pair in \cite{Br}. We note here the homogeneous associative BiHom-Yang-Baxter equation (abhYBe) is different from the one introduced by Liu, Makhlouf, Menini, Panaite in \cite{LMMP3}. The solution to nonhomogeneous abhYBe can be obtained by unitary quasitriangular covariant BiHom-bialgebra.

 \subsection{Nonhomogeneous abhYBp}
 For convenience, we introduce the following notation. Let $(A,\mu,\alpha,\beta)$ be a BiHom-associative algebra, $\psi,\omega: A\lr A$ be linear maps, $r, s\in A\otimes A$, we define the following elements in $A\otimes A\otimes A$:
 \begin{eqnarray*}
 &r_{12}s_{23}=\alpha(r^{1})\otimes r^{2}s^{1}\otimes\beta(s^{2}),~~ r_{13}s_{12}=\omega(r^{1})s^{1}\otimes \beta(s^{2})\otimes\alpha\psi(r^{2}),&\\
 &r_{23}s_{13}=\beta\omega(s^{1})\otimes\alpha(r^{1})\otimes r^{2}\psi(s^{2}).&
 \end{eqnarray*}

 \begin{defi}\mlabel{de:4.34} An {\bf associative BiHom-Yang-Baxter pair (abhYBp) of weight $(\lambda,\gamma)$ in $(A,\mu,\alpha,\beta)$} is a pair of elements $r, s\in A\otimes A$ satisfying
 \begin{eqnarray}
 &r_{13}r_{12}-r_{12}r_{23}+s_{23}r_{13}=-\lambda r_{13},&\mlabel{eq:4.53}\\
 &s_{13}r_{12}-s_{12}s_{23}+s_{23}s_{13}=-\gamma s_{13},&\mlabel{eq:4.54}
 \end{eqnarray}
 where $\lambda, \gamma\in K$.

 Specially, let $r=s$, $\lambda=\gamma$, then we call
 \begin{eqnarray}
 r_{13}r_{12}-r_{12}r_{23}+r_{23}r_{13}=-\lambda r_{13} \mlabel{eq:4.54a}
 \end{eqnarray}
 the {\bf associative BiHom-Yang-Baxter equation (abhYBe) of weight $\lambda$ in $(A,\mu,\alpha,\beta)$}.
 \end{defi}

 \begin{rmk} (1) The abhYBe of weight 0 here is different from the one introduced by Liu, Makhlouf, Menini, Panaite in \cite[Definition 4.1]{LMMP} and also studied by the same authors in \cite{LMMP3}. The formula in Eq.(\mref{eq:4.54a}) consists of four structure maps $\a, \b, \psi$ and $\om$. We will discuss the relation between them later.

 (2) If $\a=\b=\psi=\om=\id$ in Eq.(\mref{eq:4.54a}), then the nonhomogeneous associative cYBe introduced in \cite{E06} (also studied in \cite{OP10}) can be obtained and its solution can be provided by B\'{e}zout operators \cite{OP10}. Moreover, $\l=0$, then we can get the associative cYBe studied in \cite{A00} which is an associative analogy of cYBe.

 (3) If $\a=\b=\psi=\om=\id$, then the abhYBp of weight $(0,0)$ in $(A,\mu)$ is exactly the associative Yang-Baxter pair in \cite{Br}.

 (4) Eq.(\ref{eq:4.54a}) was introduced in \cite{MLi} which corresponds to the case of anti-quasitriangular unitary $\l$-infinitesimal BiHom-bialgebra.
 \end{rmk}

 \subsection{Double curved Rota-Baxter BiHom-system}

 As a generalization of Rota-Baxter algebra, Rota-Baxter system was introduced by Brzezi\'{n}ski in \cite{Br} and its curved version can be found in \cite{Br2}. Here we slightly modified the Brzezi\'{n}ski's curved Rota-Baxter system and also extend it to the BiHom-case.

 \begin{defi}\mlabel{de:4.45} A 7-tuple $(A, R, S, \alpha, \beta, \xi, \zeta)$ consisting of a BiHom-associative algebra $(A, \mu, \alpha, \beta)$, linear operators $R, S: A\lr A$, and $\xi,\zeta: A\otimes A\lr A$ is called {\bf a double curved Rota-Baxter BiHom-system}, if for all $a, b \in A$,
 \begin{eqnarray}
 &\alpha\circ R=R\circ \alpha,\quad \alpha\circ S=S\circ \alpha, \quad\beta\circ R=R\circ \beta,\quad \beta\circ S=S\circ \beta,&\mlabel{eq:4.76}\\
 &R(a)R(b)=R(R(a)b)+R(aS(b))+\xi(a\otimes b),&\mlabel{eq:4.77}\\
 &S(a)S(b)=S(R(a)b)+S(aS(b))+\zeta(a\otimes b).&\mlabel{eq:4.78}
 \end{eqnarray} 
 \end{defi}

 \begin{rmk}\mlabel{rmk:4.41B} If $\a=\b=\id$ and $\xi=\zeta$ in Definition \mref{de:4.45}, then we can get curved Rota-Baxter system introduced by Brzezi\'{n}ski in \cite{Br2}. Moreover, if $\xi=\zeta=0$, then Rota-Baxter system can be obtained which was introduced by Brzezi\'{n}ski in \cite{Br}. 
 \end{rmk}

 Let $A$ be a vector space, $F: A\lr A$ a linear map. Then an element $r\in A\otimes A$ is called $F$-{\bf invariant} if $(F\otimes F)(r)=r$.

 The abhYBp of weight $(\lambda,\gamma)$ in $(A,\mu,\alpha,\beta)$ can provide the solution to the double curved Rota-Baxter system as follows.

 \begin{thm}\mlabel{thm:4.46} Let $(A,\mu,\alpha,\beta)$ be a BiHom-associative algebra, $\psi,\omega: A\lr A$ be linear maps such that $\psi(ab)=\psi(a)\psi(b)$, $\omega(ab)=\omega(a)\omega(b)$, $\alpha,\beta,\psi,\omega$ are bijective and any two of them commute. Assume that $(r, s)$ is an abhYBp of weight $(\lambda,\gamma)$ in $(A,\mu,\alpha,\beta)$ such that $r, s$ are $\alpha,\beta,\psi,\omega$-invariant. Define $R, S: A\lr A$ by
 \begin{eqnarray}
 &R(a)=\beta^{2}\psi(r^{1})
 (\alpha^{-1}\beta^{-1}(a)\alpha\omega(r^{2})),\quad S(a)=\beta^{2}\psi(s^{1})
 (\alpha^{-1}\beta^{-1}(a)\alpha\omega(s^{2})),&\mlabel{eq:4.81}
 \end{eqnarray}
 and $\xi,\zeta: A\otimes A\lr A$ by
 \begin{eqnarray}
 &\xi(a\otimes b)=\lambda R(ab),\quad \zeta(a\otimes b)=\gamma S(ab),&\mlabel{eq:4.79}
 \end{eqnarray}
 for all $a,b\in A$. Then $(A, R, S, \alpha, \beta, \xi, \zeta)$ is a double curved Rota-Baxter BiHom-system.
 \end{thm}

 \begin{proof}
 The fact that $R,S$ commute with $\alpha$ and $\beta$ follows immediately from the fact that $r,s$ are $\alpha,\beta, \psi,\omega$-invariant. Next we check that Eq.(\mref{eq:4.77}) holds. For all $a,b\in A$,
 \begin{eqnarray*}
 R(a)R(b)
 &\stackrel{(\mref{eq:1.3})}=&\alpha\beta^{2}\psi(r^{1})((\alpha^{-1}\beta^{-1}(a)
 (\omega(r^{2})\psi(\bar{r}^{1})))(\alpha^{-1}\beta^{-1}(b)\alpha\omega(\bar{r}^{2})))\\
 &\stackrel{(\mref{eq:4.53})}=&
 (\beta^{2}\psi^{2}(r^{1})\beta^{2}\psi^{2}\omega^{-1}(\bar{r}^{1}))
 ((\alpha^{-1}\beta^{-1}(a)
 \psi\beta(\bar{r}^{2}))(\alpha^{-1}\beta^{-1}(b)\alpha^{2}\beta^{-1}\psi\omega(r^{2})))\\
 &&+\beta^{3}\psi^{2}(r^{1})((\alpha^{-1}\beta^{-1}(a)
 \alpha\psi(s^{1}))(\alpha^{-1}\beta^{-1}(b)(\alpha\beta^{-1}\omega(s^{2})\alpha\beta^{-1}\omega\psi(r^{2}))))\\
 &&+\lambda\beta^{2}\psi^{2}(r^{1})((\alpha^{-1}\beta^{-1}(a)
 \cdot 1)(\alpha^{-1}\beta^{-1}(b)\alpha\beta^{-1}\omega\psi(r^{2})))\\
 &\stackrel{(\mref{eq:1.3})}=&\beta^{2}\psi(\bar{r}^{1})
 ((\beta\psi(r^{1})((\alpha^{-2}\beta^{-2}(a)\beta^{-1}\omega(r^{2}))(\alpha^{-1}\beta^{-2}(b))))
 \alpha\omega(\bar{r}^{2}))\\
 &&+\beta^{2}\psi(r^{1})(((\alpha^{-2}\beta^{-1}(a)\alpha^{-1}\beta\psi(s^{1})) (\alpha^{-2}\beta^{-1}(b)\omega(s^{2})))\alpha\omega(r^{2}))\\
 &&+\lambda\beta^{2}\psi(r^{1})((\alpha^{-1}\beta^{-1}(a)\alpha^{-1}\beta^{-1}(b))  \alpha\omega(r^{2}))\\
 &=&R(R(a)b)+R(aS(b))+\lambda R(ab).
 \end{eqnarray*}
 Similarly, we can obtain Eq.(\mref{eq:4.78}). These finish the proof.
 \end{proof} 

 \begin{cor}\mlabel{cor:4.46a} Let $(A,\mu,\alpha,\beta)$ be a BiHom-associative algebra, $\psi,\omega: A\lr A$ be linear maps such that $\psi(ab)=\psi(a)\psi(b)$, $\omega(ab)=\omega(a)\omega(b)$, $\alpha,\beta,\psi,\omega$ are bijective and any two of them commute. Assume that $r$ is a solution to abhYBe of weight $\l$ in $(A,\mu,\alpha,\beta)$ such that $r$ is $\alpha,\beta,\psi,\omega$-invariant. Define $R: A\lr A$ by
 \begin{eqnarray*}
 &R(a)=\beta^{2}\psi(r^{1})(\alpha^{-1}\beta^{-1}(a)\alpha\omega(r^{2})).&
 \end{eqnarray*}
 Then $(A, \mu, R, \alpha, \beta)$ is a Rota-Baxter BiHom-associative algebra of wight $\l$.
 \end{cor}

 \begin{proof} Based on the proof of Theorem \mref{thm:4.46}, when $r=s$ and $\l=\g$, we can get that $R$ is a Rota-Baxter operator of weight $\l$.
 \end{proof}

 \begin{rmk}
 If $\a=\b=\psi=\om=\id$ in Corollary \mref{cor:4.46a}, then we have \cite[Proposition 3.23]{E06}.
 \end{rmk}

 \begin{cor}\mlabel{cor:4.46b} (\cite[Proposition 3.4]{Br}) Let $(A,\mu)$ be an associative algebra, $(r, s)$ be an aYBp of weight $(0,0)$ in $(A,\mu)$. Define $R, S: A\lr A$ by
 \begin{eqnarray*}
 &R(a)=r^{1}ar^{2},\quad S(a)=s^{1}as^{2}.&
 \end{eqnarray*}
 Then $(A, R, S)$ is a Rota-Baxter system.
 \end{cor}

 \begin{proof} Let $\a=\b=\psi=\om=\id$ and $\l=\g=0$ in Theorem \mref{thm:4.46}.
 \end{proof}

 \begin{lem}\mlabel{lem:8.1} (1) Let $(A, R, S, \alpha,\beta, \xi, \zeta)$ be a double curved Rota-Baxter BiHom-system and for all $a,b\in A$ define
 \begin{eqnarray}
 &a\ast b= R(a)b+ aS(b).&\mlabel{eq:8.1}
 \end{eqnarray}
 Then $(A, \ast, \alpha, \beta)$ is a BiHom-associative algebra if and only if, for all $a,b,c\in A$
 \begin{eqnarray}
 &\alpha(a)\zeta(b\otimes c)=\xi(a\otimes b)\beta(c).&\mlabel{eq:8.2t}
 \end{eqnarray}

 (2) In particular, if $(A, \ast, \alpha, \beta)$ is unitary (i.e., $\exists$ $1_A\in A$ such that Eq.(\mref{eq:1.5}) holds for $(A, \ast, \alpha, \beta))$, then $(A, \ast, \alpha, \beta)$ is a BiHom-associative algebra if and only if, for all $a, b\in A$,
 \begin{eqnarray}
 &\zeta(a\otimes b)=\a^{-1}\zeta(a\otimes 1)b,&\mlabel{eq:8.14a}\\
 &\xi(a\otimes b)=a\b^{-1}\xi(1\otimes b),&\mlabel{eq:8.14b}\\
 &\b\zeta(a\otimes 1)=\a\xi(1\otimes a).&\mlabel{eq:8.14}
 \end{eqnarray}
 \end{lem}

 \begin{proof} For Part (1), we only need to check that the equation below. For all $a, b, c\in A$, we have
 \begin{eqnarray*}
 (a\ast b)\ast \beta(c)-\alpha(a)\ast (b \ast c)
 &\stackrel{(\mref{eq:8.1})}=&R(R(a)b+ aS(b))\beta(c)+(R(a)b+ aS(b))S(\beta(c))\\
 &&-R(\alpha(a))(R(b)c+ bS(c))-\alpha(a)S(R(b)c+ bS(c))\\
 &\stackrel{(\mref{eq:4.77})(\mref{eq:4.78})}=&(R(a)R(b))\beta(c)-\xi(a\otimes b)\beta(c) +(R(a)b)\beta S(c)\\
 &&+(aS(b))\beta S(c)-\alpha R(a)(R(b)c)-\alpha R(a)(bS(c))\\
 &&-\alpha (a)(S(b)S(c))+ \alpha (a)\zeta(b\otimes c)\\
 &\stackrel{(\mref{eq:1.3})}=&\alpha (a)\zeta(b\otimes c)-\xi(a\otimes b)\beta(c),
 \end{eqnarray*}
 as desired.

 If $(A, \ast, \alpha, \beta)$ is unitary, then by Eq.(\mref{eq:8.2t}), for all $a,b\in A$, we have
 $$
 \left.
 \begin{matrix}
 \b\zeta(a\otimes b)=\xi(1\otimes a)\beta(b)\\
 \a\xi(a\otimes b)=\alpha(a)\zeta(b\otimes 1)\\
 \end{matrix}
 \right\}\Rightarrow
 \left\{
 \begin{matrix}
 \zeta(a\otimes b)=\b^{-1}\xi(1\otimes a)b=\a^{-1}\zeta(a\otimes 1)b\\
 \xi(a\otimes b)=a\a^{-1}\zeta(b\otimes 1)=a\b^{-1}\xi(1\otimes b)\\
 \end{matrix}
 \right. .
 $$
 Thus Eqs.(\ref{eq:8.14a}) and (\ref{eq:8.14b}) hold. Let $a=c=1$ in Eq.(\mref{eq:8.2t}), then we can obtain Eq.(\ref{eq:8.14}).

 Conversely, we calculate
 \begin{eqnarray*}
 \alpha (a)\zeta(b\otimes c)&\stackrel{(\mref{eq:8.14a})}=&\alpha (a)(\alpha^{-1}\zeta(b\otimes 1)c)\stackrel{(\mref{eq:1.3})}=(a\alpha^{-1}\zeta(b\otimes 1))\beta(c)\\
 &\stackrel{(\mref{eq:8.14})}=&(a\beta^{-1}\xi(1\otimes b))\beta(c)\stackrel{(\mref{eq:8.14b})}=\xi(a\otimes b)\beta(c).
 \end{eqnarray*}
 Then Eq.(\mref{eq:8.2t}) holds. Therefore we prove the second part.
 \end{proof}

 \begin{thm}\mlabel{thm:8.1a} (1) Let $(A, R, S, \alpha,\beta, \xi)$ be a curved Rota-Baxter BiHom-system and for all $a, b\in A$ define $\ast$ by Eq.(\ref{eq:8.1}).
 Then $(A, \ast, \alpha, \beta)$ is a BiHom-associative algebra if and only if, for all $a,b,c\in A$
 \begin{eqnarray*}
 &\alpha(a)\xi(b\otimes c)=\xi(a\otimes b)\beta(c).&
 \end{eqnarray*}

 (2) In particular, if $(A, \ast, \alpha, \beta)$ is unitary, then $(A, \ast, \alpha, \beta)$ is a BiHom-associative algebra if and only if, there exists an element $\kappa\in A$ such that for all $a, b\in A$,
 \begin{eqnarray}
 &\a(\kappa)=\b(\kappa),&\mlabel{eq:8.14c}\\
 &\kappa\a(a)=\b(a)\kappa,&\mlabel{eq:8.14d}\\
 &\xi(a\otimes b)=\b^{-1}(\k)\b^{-1}(ab).&\mlabel{eq:8.14e}
 \end{eqnarray}
 \end{thm}

 \begin{proof} Part (1) is obvious by letting $\xi=\zeta$ in Lemma \ref{lem:8.1}. For Part (2), based on Part (2) in Lemma \ref{lem:8.1} and Part (1), we need check that Eqs.(\ref{eq:8.14c})-(\ref{eq:8.14e}) are equivalent to Eqs.(\ref{eq:8.14b}), (\ref{eq:8.14}) and
 \begin{eqnarray}
 &\xi(a\otimes b)=\a^{-1}\xi(a\otimes 1)b.&\mlabel{eq:8.14f} 
 \end{eqnarray}

 $(\Longrightarrow)$ Firstly, for all $a, b\in A$, we verify that Eq.(\mref{eq:8.14f}) holds as follows.
 \begin{eqnarray*}
 \a^{-1}\xi(a\otimes 1)b
 &\stackrel{(\mref{eq:8.14e})}=&\a^{-1}(\b^{-1}(\k)\b^{-1}(a1))b\\
 &\stackrel{(\mref{eq:1.5})}=&\a^{-1}(\b^{-1}(\k)\b^{-1}\a(a))b\stackrel{(\mref{eq:1.3})}=\xi(a\otimes b).
 \end{eqnarray*}
 By Eqs.(\mref{eq:8.14c}) and (\mref{eq:8.14d}), one gets
 \begin{eqnarray}
 &\b^{-1}(\k)\b^{-1}(ab)=\a^{-1}(ab)\a^{-1}(\k).& \mlabel{eq:8.14i}
 \end{eqnarray}
 So by Eq.(\mref{eq:8.14e}), we have
 \begin{eqnarray}
 &\xi(a\otimes b)=\a^{-1}(ab)\a^{-1}(\k).&\mlabel{eq:8.14h}
 \end{eqnarray}
 Then
 \begin{eqnarray*}
 a\b^{-1}\xi(1\otimes b)
 &\stackrel{(\mref{eq:8.14h})}=&a\b^{-1}(\a^{-1}(1b)\a^{-1}(\k))\\
 &\stackrel{(\mref{eq:1.5})}=&a\b^{-1}(\a^{-1}\b(b)\a^{-1}(\k))\\
 &\stackrel{(\mref{eq:1.3})}=&\a^{-1}(ab)\a^{-1}(\k)
 \stackrel{(\mref{eq:8.14i})}=\xi(a\otimes b),\\
 \b\xi(a\otimes 1)
 &\stackrel{(\mref{eq:8.14e})}=&\b(\b^{-1}(\k)\b^{-1}(a1))\stackrel{(\mref{eq:1.5})}=\k\a(a)\\
 &\stackrel{(\mref{eq:8.14d})}=&\b(a)\k\stackrel{(\mref{eq:1.5})}=\a(\a^{-1}(1a)\a^{-1}(\k)) \stackrel{(\mref{eq:8.14h})}=\a\xi(1\otimes a).
 \end{eqnarray*}
 Thus Eqs.(\ref{eq:8.14b}) and (\ref{eq:8.14}) hold.

 $(\Longleftarrow)$ Set $\k=\xi(1\o 1)\in A$. Let $a=1$ in Eq.(\ref{eq:8.14}), then we obtain Eq.(\ref{eq:8.14c}).
 By Eqs.(\ref{eq:8.14}) and (\ref{eq:8.14f}), then
 \begin{eqnarray}
 &\xi(a\otimes b)=\b^{-1}\xi(1\o a)b.&\mlabel{eq:8.14j}
 \end{eqnarray}
 Thus Eq.(\ref{eq:8.14e}) can be proved by Eqs.(\ref{eq:8.14c}) and (\ref{eq:8.14j}).

 Similarly, by Eqs.(\ref{eq:8.14b}), (\ref{eq:8.14}) and (\ref{eq:8.14c}), we can get
 \begin{eqnarray}
 &\xi(a\otimes b)=\a^{-1}(ab)\a^{-1}(\k).&\mlabel{eq:8.14k}
 \end{eqnarray}

 Then Eq.(\ref{eq:8.14d}) can be obtained by the following calculation.
 \begin{eqnarray*}
 \k\a(a)
 &\stackrel{(\mref{eq:1.5})}=&\b(\b^{-1}(\k)\b^{-1}(a1))
 \stackrel{(\mref{eq:8.14e})}=\b\xi(a\o 1)\\
 &\stackrel{(\mref{eq:8.14})}=&\a\xi(1\otimes a)
 \stackrel{(\mref{eq:8.14k})}=\a(\a^{-1}(1a)\a^{-1}(\k)))\stackrel{(\mref{eq:1.5})}=\b(a)\k,
 \end{eqnarray*}
 finishing the proof.
 \end{proof}

 \begin{rmk}\mlabel{rmk:8.1aa} Let $\a=\b=\id$ in Theorem \ref{thm:8.1a}, then we can get \cite[Proposition 2.1]{Br2}. In this case, Eq.(\ref{eq:8.14d}) implies that $\k$ lies in the center of $A$.
 \end{rmk}

 \subsection{Double curved weak BiHom-pseudotwistor} The notions of a twistor and more generally a pseudotwistor were introduced by L\'{o}pez Pe\~{n}a, Panaite and Van Oystaeyen \cite{LPPVO} to provide a framework unifying some of the deformed or twisted algebras. The more general concept of a weak pseudotwistor was introduced by Panaite and Van Oystaeyen in \cite{PVO}, then the curved version was studied in \cite{Br2}. Here we provide the double curved (BiHom-)version.

 \begin{defi}\mlabel{de:8.2} Let $(A,\mu,\alpha,\beta)$ be a BiHom-associative algebra. A $K$-linear map $T:A\otimes A\lr  A\otimes A$ is called a {\bf double curved weak BiHom-pseudotwistor} if there exist $K$-linear maps $\mathfrak{T}:A\otimes A\otimes A\lr  A\otimes A\otimes A$ and $\xi,\zeta:A\otimes A\lr  A$, such that the following conditions hold
 \begin{eqnarray}
 &\mu(\alpha\o \zeta)=\mu(\xi\o \beta),&\mlabel{eq:8.2}\\
 &T(\alpha\otimes\mu T)=(\alpha\otimes\mu)\mathfrak{T}-\alpha\otimes\zeta,&\mlabel{eq:8.3}\\
 &T(\mu T\otimes\beta)=(\mu\otimes\beta)\mathfrak{T}-\xi\otimes\beta.&\mlabel{eq:8.4}
 \end{eqnarray}
 The map $\mathfrak{T}$ is called a {\bf weak BiHom-companion of} $T$ and $\xi,\zeta$ are called {\bf curvatures of} $T$.
 \end{defi}

 \begin{rmk}\mlabel{rmk:8.8} Let $\xi=\zeta$ in Definition \mref{de:8.2}, then $T$ is called a {\bf curved weak BiHom-pseudotwistor}. And if furthermore $\alpha=\beta=\id$, then we can obtain curved weak pseudotwistor introduced in \cite[Definition 2.2]{Br2}.
 \end{rmk}

 \begin{pro}\mlabel{pro:8.3} Let $(A,\mu,\alpha,\beta)$ be a BiHom-associative algebra and $T:A\otimes A\lr  A\otimes A$ a double curved weak BiHom-pseudotwistor with companion $\mathfrak{T}$ and curvatures $\xi, \zeta$. Then $(A, \mu T, \alpha,\beta)$ is a BiHom-associative algebra.
 \end{pro}

 \begin{proof} We calculate BiHom-associativity as follows:
 \begin{eqnarray*}
 \mu T (\alpha \otimes\mu  T)
 &\stackrel{(\mref{eq:8.3})}=&\mu(\alpha\otimes\mu)\mathfrak{T}-\mu(\alpha\otimes\zeta)\\
 &\stackrel{(\mref{eq:8.2})(\mref{eq:1.3})}=&\mu(\mu\otimes\beta)\mathfrak{T}-\mu(\xi\o \beta)\stackrel{(\mref{eq:8.4})}=\mu T(\mu T\otimes\beta),
 \end{eqnarray*}
 as required.
 \end{proof}

 \begin{rmk}\mlabel{rmk:8.9} Let $\xi=\zeta$ and $\alpha=\beta=\id$ in Proposition \mref{pro:8.3}, then the associative algebra can be obtained by curved weak pseudotwistor introduced in \cite[Lemma 2.3]{Br2}.
 \end{rmk}

 \begin{thm}\mlabel{thm:8.4} Let $(A,R,S,\alpha,\beta,\xi,\zeta)$ be a double curved Rota-Baxter BiHom-system that satisfies Eq.(\mref{eq:8.2}), and define, for all $a,b,c\in A$
 \begin{eqnarray}
 &T(a\otimes b)=R(a)\otimes b+a\otimes S(b) &\mlabel{eq:8.6}
  \end{eqnarray}
 and
 \begin{eqnarray}
 &\mathfrak{T}(a\otimes b\otimes c)=R(a)\otimes R(b)\otimes c+ R(a)\otimes b\otimes S(c)+ a\otimes S(b)\otimes S(c).&\mlabel{eq:8.7}
 \end{eqnarray}
 Then $T$ is a double curved BiHom-pseudotwistor with the weak companion $\mathfrak{T}$ and curvatures $\xi,\zeta$.
 \end{thm}

 \begin{proof} For all $a,b,c\in A$, we compute 
 \begin{eqnarray*}
 T(\alpha\otimes\mu T)(a\otimes b\otimes c) 
 &\stackrel{(\mref{eq:8.6})}=&R(\alpha(a))\otimes R(b)c+\alpha(a)\otimes S(R(b)c) +R(\alpha(a))\otimes bS(c)\\
 &&+\alpha(a) \otimes S(bS(c))\\
 &\stackrel{(\mref{eq:4.78})}=&\alpha R(a)\otimes R(b)c+\alpha R(a)\otimes bS(c)+\alpha(a)\otimes S(b)S(c)\\
 &&-\alpha(a) \otimes \zeta(b\otimes c)\\ 
 &\stackrel{(\mref{eq:8.7})}=&((\alpha\otimes\mu)\mathfrak{T}-\alpha\otimes\zeta)(a\otimes b\otimes c).
 \end{eqnarray*}
 In a similar way we can get
 \begin{eqnarray*}
 &&T(\mu T\otimes\beta)(a\otimes b\otimes c) 
 =((\mu\otimes\beta)\mathfrak{T}-\xi\otimes\beta)(a\otimes b\otimes c).
 \end{eqnarray*}
 Therefore we finish the proof.
 \end{proof}

 \begin{rmk}\mlabel{rmk:8.10} Let $\xi=\zeta$ and $\alpha=\beta=\id$ in Theorem \mref{thm:8.4}, then curved pseudotwistor can be obtained by curved Rota-Baxter system introduced in \cite[Lemma 2.4]{Br2}.
 \end{rmk}

 \subsection{Unital quasitriangular covariant BiHom-bialgebra}\label{se:tq} In this subsection, we mainly provide a construction of nonhomogeneous abhYBe through unitary quasitriangular covariant BiHom-bialgebra. Here we really present the BiHom-type (including four structure maps $\a, \b, \psi, \om$) of quasitriangular infinitesimal bialgebra, which is a generalization of Liu, Makhlouf, Menini, Panaite's BiHom-type in \cite[Section 5]{LMMP3}.

 \begin{pro}\mlabel{pro:2.1A} Let $(A, \mu, \alpha_{A}, \beta_{A})$ be a BiHom-associative algebra and $(M, \triangleright, \triangleleft, \alpha_{M}, \beta_{M})$, $(N, \triangleright, \triangleleft,$ $\alpha_{N}, \beta_{N})$, $(V, \triangleright, \triangleleft, \alpha_{V}, \beta_{V})$ be $(A, \mu, \alpha_{A}, \beta_{A})$-bimodules. $\psi_{A}, \omega_{A}: A\lr A$ are linear maps such that for all $a,b\in A,$ $\psi_{A}(ab)=\psi_{A}(a)\psi_{A}(b)$, $\omega_{A}(ab)=\omega_{A}(a)\omega_{A}(b)$ and any two of the maps $\alpha_{A}$, $\beta_{A}$, $\psi_{A}$, $\omega_{A}$ commute. We consider the following left and the right action of $A$ on $M\otimes N\otimes V$, for all $a\in A, m\in M, n\in N, v\in V$
 \begin{eqnarray*}
 &a\triangleright(m\otimes n\otimes v)=\omega_{A}(a)\triangleright m\otimes\beta_{N}(n)\otimes\beta_{V}(v),&\mlabel{eq:2.1A}\\
 &(m\otimes n\otimes v)\triangleleft a=\alpha_{M}(m)\otimes \alpha_{N}(n)\otimes v\triangleleft\psi_{A}(a).&\mlabel{eq:2.2A}
 \end{eqnarray*}
 Then $(M\otimes N\otimes V, \triangleright, \triangleleft, \alpha_{M}\otimes\alpha_{N}\otimes\alpha_{V}, \beta_{M}\otimes\beta_{N}\otimes\beta_{V})$ is an $(A,\mu,\alpha_{A},\beta_{A})$-bimodule.
 \end{pro}

 \begin{proof} 
 We only check the compatible condition between left and right action as follows, the rest can be proved similarly.
 \begin{eqnarray*}
 \alpha_{A}(a)\triangleright((m\otimes n\otimes v)\triangleleft b) 
 &=&\omega_{A}\alpha_{A}(a)\triangleright \alpha_{M}(m)\otimes\alpha_{N}\beta_{N}(n)\otimes\beta_{V}(v)\triangleleft \psi_{A}\beta_{A}(b)\\
 &\stackrel{(\mref{eq:1.13})}=&\alpha_{M}(\omega_{A}(a)\triangleright m)\otimes\alpha_{N}\beta_{N}(n)\otimes \beta_{V}(v)\triangleleft \psi_{A}\beta_{A}(b)\\ 
 &=&(a\triangleright(m\otimes n\otimes v))\triangleleft\beta_{A}(b),
 \end{eqnarray*}
 finishing the proof. 
 \end{proof}

 \begin{rmk}\mlabel{rmk:2.2} (1) As the special cases of Proposition \mref{pro:2.1A}, we have the following actions of $A$ on $A\otimes A$ and $A\otimes A\otimes A$ by
 \begin{eqnarray}
 &a\triangleright(x\otimes y)=\omega(a)x\otimes\beta(y),&\mlabel{eq:2.3}\\
 &(x\otimes y)\triangleleft a=\alpha(x)\otimes y\psi(a),&\mlabel{eq:2.4}
 \end{eqnarray}
 and
 \begin{eqnarray}
 &a\triangleright(x\otimes y\otimes z)=\omega(a)x\otimes\beta(y)\otimes\beta(z),&\mlabel{eq:2.5}\\
 &(x\otimes y\otimes z)\triangleleft a=\alpha(x)\otimes\alpha(y)\otimes z\psi(a).&\mlabel{eq:2.6}
 \end{eqnarray}
 respectively.

 (2) The actions defined in Eqs.(\mref{eq:2.3}) and (\mref{eq:2.4}) coincide with the actions in \cite[Lemma 4.2]{LMMP3}.
 \end{rmk}

 \begin{defi}\mlabel{de:2.3} Let $(A,\mu,\alpha,\beta)$ be a BiHom-associative algebra and $\psi,\omega: A\lr A$ be linear maps, then $\delta: A\lr A\otimes A$ 
 is a {\bf BiHom-derivation} if
 \begin{eqnarray}
 &(\alpha\otimes\alpha)\circ\delta=\delta\circ\alpha, ~~(\beta\otimes\beta)\circ\delta=\delta\circ\beta,
 ~~(\psi\otimes\psi)\circ\delta=\delta\circ\psi,
 ~~(\omega\otimes\omega)\circ\delta=\delta\circ\omega,&\mlabel{eq:2.7}\\
 &\delta(ab)=a\triangleright\delta(b)+\delta(a)\triangleleft b.&\mlabel{eq:2.8}
 \end{eqnarray} 
 \end{defi}

 \begin{defi}\mlabel{de:2.5} \cite{MY} A {\bf covariant BiHom-bialgebra} is a 9-tuple $(A, \mu, \delta_{1}, \delta_{2},$ $\Delta, \alpha, \beta, \psi, \omega)$ where $\d_i: A\lr A\o A, i=1,2$ are two BiHom-derivations such that

 (1) $(A, \mu, \alpha, \beta)$ is a BiHom-associative algebra,

 (2) $(A, \Delta, \psi, \omega)$ is a BiHom-coassociative coalgebra,

 (3) $\Delta$ is covariant BiHom-derivation with respect to BiHom-derivations $\delta_{1},\delta_{2}$, i.e.,
 \begin{eqnarray*}
 \Delta\circ\mu
 &=&\triangleleft\circ(\Delta\otimes \id)+\triangleright\circ(\id\otimes \delta_{1})\\
 &=&\triangleright\circ(\id\otimes\Delta)+\triangleleft\circ(\delta_{2}\otimes \id)
 \end{eqnarray*}
 and moreover the following relations are satisfied
\begin{eqnarray}
 &\alpha\circ\psi=\psi\circ\alpha,~~\alpha\circ\omega=\omega\circ\alpha,~~ \beta\circ\psi=\psi\circ\beta,~~\beta\circ\omega=\omega\circ\beta,&\mlabel{eq:2.12}\\
 &(\alpha\otimes\alpha)\circ\Delta=\Delta\circ\alpha,~~ (\beta\otimes\beta)\circ\Delta=\Delta\circ\beta,&\mlabel{eq:2.13}\\
 &\psi\circ\mu=\mu\circ(\psi\otimes \psi),~~\omega\circ\mu=\mu\circ(\omega\otimes \omega).&\mlabel{eq:2.14}
 \end{eqnarray}

 If $(A,\mu,\alpha,\beta)$ is unitary with unit $1_{A}$ such that
 \begin{eqnarray}
 &\psi(1_{A})=1_{A}\hbox{~and~}\omega(1_{A})=1_{A},&\mlabel{eq:2.16}
 \end{eqnarray} then a covariant BiHom-bialgebra  $(A,\mu,\delta_{1},\delta_{2},\Delta,\alpha,\beta,\psi,\omega)$ is said to be {\bf unitary} if $\Delta(1_A)=1_A\otimes 1_A$ holds.
 \end{defi}

 \begin{rmk} (1) Let $\alpha=\beta=\psi=\omega=\id$, then we can get covariant bialgebra  introduced in \cite{Br}.

 (2) Let $\delta_1=\delta_2=\Delta$ be BiHom-derivations. Then we obtain the BiHom-type of infinitesimal bialgebra introduced in \cite{LMMP3}, i.e., an {\bf infinitesimal BiHom-bialgebra} is a 7-tuple $(A, \mu, \Delta, \alpha, \beta, \psi, \omega)$, with the property that $(A, \mu, \alpha, \beta)$ is a BiHom-associative algebra, $(A, \Delta, \psi,\omega)$ is a BiHom-coassociative coalgebra, Eqs.(\mref{eq:2.12})-(\mref{eq:2.14}) and the following condition are satisfied
 \begin{eqnarray*}
 \Delta \circ \mu=(\mu\o \beta)\circ (\om \o \Delta)+(\alpha\o \mu)\circ(\Delta\o \psi). 
 \end{eqnarray*}
 
 \end{rmk}

 Below we give an equivalent characterization for covariant BiHom-bialgebra.

 \begin{thm}\mlabel{thm:2.9} Let $(A,\mu,\alpha,\beta)$ be a unitary BiHom-associative algebra (where $\alpha,\beta$ are bijective), and $\psi,\omega: A\lr A$ be two linear maps, and $\alpha,\beta,\psi,\omega $ satisfy Eqs.(\mref{eq:2.12})-(\mref{eq:2.16}). Let $\delta_{1},\delta_{2}:A\lr A\otimes A$ be two BiHom-derivations, $\Delta:A\lr A\otimes A$ be a linear map such that Eq.(\ref{eq:1.7}) holds. Then $\Delta$ is a coassociative covariant BiHom-derivation with respect to $(\delta_{1},\delta_{2})$ if and only if there exists an element $u=u^{1}\otimes u^{2}\in A\otimes A$ such that for all $a\in A$,
 \begin{eqnarray}
 &(\delta_{1}-\delta_{2})(a)=\alpha^{-1}(a)\triangleright u-u\triangleleft\beta^{-1}(a),&\mlabel{eq:2.27}\\
 &(\delta_{1}\otimes\psi-\omega\otimes\delta_{1})\circ\delta_{1}(a)=
 u_{23}\delta_{1}(a)_{13},&\mlabel{eq:2.28}\\
 &(\delta_{1}\otimes\psi-\omega\otimes\delta_{1})(u)=u_{23}u_{13}-u_{12}u_{23}.&\mlabel{eq:2.29}
 \end{eqnarray}
 In this case,
 \begin{eqnarray}
 \Delta(a)&=&\alpha(u^{1})\otimes u^{2}\psi\beta^{-1}(a)+\delta_{1}(a),\nonumber\\
 &=&\omega\alpha^{-1}(a)u^{1}\otimes\beta(u^{2})+\delta_{2}(a).\mlabel{eq:2.30}
 \end{eqnarray}
 \end{thm}

 \begin{proof} Let $\Delta:A\lr A\otimes A$ be a linear map and set $u=\Delta(1)\ (=1_{1}\otimes 1_{2}=\bar{1}_{1}\otimes \bar{1}_{2})$ and write $\delta_{i}(a)=a^{i}_{[1]}\otimes a^{i}_{[2]}, i=1,2$.

 $(\Longrightarrow)$ Assume that $\Delta$ is a covariant BiHom-derivation, then, for all $ a\in A$, by Eqs.(\mref{eq:1.5}), (\mref{eq:2.7}) and (\mref{eq:2.16}), we have
 \begin{eqnarray*}
 \Delta(a)&=&\alpha(1_{1})\otimes 1_{2}\psi\beta^{-1}(a)+a^{1}_{[1]}\otimes a^{1}_{[2]}\\
 &=&\omega\alpha^{-1}(a)1_{1}\otimes\beta(1_{2})+a^{2}_{[1]}\otimes a^{2}_{[2]},
 \end{eqnarray*}
 which is exactly Eqs.(\mref{eq:2.27}) and (\mref{eq:2.30}). While
 \begin{eqnarray*}
 (\omega\otimes\Delta)\Delta(a) 
 &\stackrel{(\mref{eq:2.30})}=&\omega\alpha(1_{1})\otimes\alpha(\bar{1}_{1})
 \otimes\bar{1}_{2}\psi\beta^{-1}(1_{2}\psi\beta^{-1}(a))\\
 &&+\omega\alpha(1_{1})\otimes(1_{2}\psi\beta^{-1}(a))^{1}_{[1]}
 \otimes(1_{2}\psi\beta^{-1}(a))^{1}_{[2]}\\
 &&+\omega(a^{1}_{[1]})\otimes\alpha(1_{1})\otimes 1_{2}\psi\beta^{-1}(a^{1}_{[2]})\\
 &&+\omega(a^{1}_{[1]})\otimes(a^{1}_{[2]})^{1}_{[1]}\otimes (a^{1}_{[2]})^{1}_{[2]}\\
 &\stackrel{(\mref{eq:2.8})(\mref{eq:2.14})(\mref{eq:1.2})}=&\omega\alpha(1_{1})\otimes\alpha(\bar{1}_{1})
 \otimes\bar{1}_{2}(\psi\beta^{-1}(1_{2})\psi^{2}\beta^{-2}(a))\\
 &&+\omega\alpha(1_{1})\otimes\omega(1_{2})\psi\beta^{-1}(a)^{1}_{[1]}
 \otimes\beta(\psi\beta^{-1}(a)^{1}_{[2]})\\
 &&+\omega\alpha(1_{1})\otimes\alpha((1_{2})^{1}_{[1]})\otimes (1_{2})^{1}_{[2]}\psi^{2}\beta^{-1}(a)\\
 &&+\omega(a^{1}_{[1]})\otimes\alpha(1_{1})\otimes 1_{2}\psi\beta^{-1}(a^{1}_{[2]})\\
 &&+\omega(a^{1}_{[1]})\otimes(a^{1}_{[2]})^{1}_{[1]}\otimes(a^{1}_{[2]})^{1}_{[2]}
 \end{eqnarray*}
 and
 \begin{eqnarray*}
 (\Delta\otimes\psi)\Delta(a)&\stackrel{(\mref{eq:2.30})(\mref{eq:2.14})}=& 
 \alpha(\bar{1}_{1})\otimes \bar{1}_{2}\psi\beta^{-1}\alpha(1_{1})\otimes\psi(1_{2})\psi^{2}\beta^{-1}(a)\\
 &&+\alpha(1_{1})^{1}_{[1]}\otimes\alpha(1_{1})^{1}_{[2]}\otimes\psi(1_{2})\psi^{2}\beta^{-1}(a)\\
 &&+\alpha(1_{1})\otimes 1_{2}\psi\beta^{-1}(a^{1}_{[1]})\otimes\psi(a^{1}_{[2]})\\
 &&+(a^{1}_{[1]})^{1}_{[1]}\otimes(a^{1}_{[1]})^{1}_{[2]}\otimes\psi(a^{1}_{[2]}).
 \end{eqnarray*}
 Then by the BiHom-coassociative of $\Delta$ at $a=1$, and %s.(\mref{eq:1.4}),
 Eq.(\mref{eq:2.16}) we can get
 \begin{eqnarray*}
 &&\omega\alpha(1_{1})\otimes\alpha(\bar{1}_{1})
 \otimes\bar{1}_{2}(\psi\beta^{-1}(1_{2})\cdot1)
 +\omega\alpha(1_{1})\otimes\omega(1_{2})1^{1}_{[1]}
 \otimes\beta(1^{1}_{[2]})\\
 &&+\omega\alpha(1_{1})\otimes\alpha((1_{2})^{1}_{[1]})\otimes (1_{2})^{1}_{[2]}\cdot1+\omega(1^{1}_{[1]})\otimes\alpha(1_{1})\otimes 1_{2}\psi\beta^{-1}(1^{1}_{[2]})\\
 &&+\omega(1^{1}_{[1]})\otimes(1^{1}_{[2]})^{1}_{[1]}\otimes(1^{1}_{[2]})^{1}_{[2]}\\
 &=&\alpha(\bar{1}_{1})\otimes \bar{1}_{2}\psi\beta^{-1}\alpha(1_{1})\otimes\psi(1_{2})\cdot1
 +\alpha(1_{1})^{1}_{[1]}\otimes\alpha(1_{1})^{1}_{[2]}\otimes\psi(1_{2})\cdot1\\
 &&+\alpha(1_{1})\otimes 1_{2}\psi\beta^{-1}(1^{1}_{[1]})\otimes\psi(1^{1}_{[2]})
 +(1^{1}_{[1]})^{1}_{[1]}\otimes(1^{1}_{[1]})^{1}_{[2]}\otimes\psi(1^{1}_{[2]}).
 \end{eqnarray*}
 Based on the properties of BiHom-derivations and covariant BiHom-derivations, one can obtain
 \begin{eqnarray*}
 &&(1_{1})^{1}_{[1]}\otimes(1_{1})^{1}_{[2]}\otimes\psi(1_{2})
 -\omega(1_{1})\otimes(1_{2})^{1}_{[1]}\otimes (1_{2})^{1}_{[2]}\\
 &=&\omega(1_{1})\otimes\alpha(\bar{1}_{1})
 \otimes\bar{1}_{2}\psi\beta^{-1}(1_{2})-\alpha(\bar{1}_{1})\otimes \bar{1}_{2}\beta^{-1}(1_{1})\otimes(1_{2}),
 \end{eqnarray*}
 which is equivalent to Eq.(\mref{eq:2.29}).

 Finally, by the BiHom-coassociativity for $\Delta$ at any $a\in A$, one has
 \begin{eqnarray*}
 &&\omega\alpha(1_{1})\otimes\alpha(\bar{1}_{1})
 \otimes\bar{1}_{2}(\psi\beta^{-1}(1_{2})\psi^{2}\beta^{-2}(a))
 +\omega\alpha(1_{1})\otimes\alpha((1_{2})^{1}_{[1]})\otimes (1_{2})^{1}_{[2]}\psi^{2}\beta^{-1}(a)\\
 &&+\omega(a^{1}_{[1]})\otimes\alpha(1_{1})\otimes 1_{2}\psi\beta^{-1}(a^{1}_{[2]})
 +\omega(a^{1}_{[1]})\otimes(a^{1}_{[2]})^{1}_{[1]}\otimes(a^{1}_{[2]})^{1}_{[2]}\\
 &=&\alpha(\bar{1}_{1})\otimes \bar{1}_{2}\psi\beta^{-1}\alpha(1_{1})\otimes\psi(1_{2})\psi^{2}\beta^{-1}(a)
 +\alpha(1_{1})^{1}_{[1]}\otimes\alpha(1_{1})^{1}_{[2]}\otimes\psi(1_{2})\psi^{2}\beta^{-1}(a)\\
 &&+(a^{1}_{[1]})^{1}_{[1]}\otimes(a^{1}_{[1]})^{1}_{[2]}\otimes\psi(a^{1}_{[2]}).
 \end{eqnarray*}
 Again by the properties of BiHom-derivations, covariant BiHom-derivations and Eq.(\mref{eq:2.29}), we find that the condition of BiHom-coassociativity for $\Delta$ at any $a\in A$ is exactly Eq.(\mref{eq:2.28}).

 $(\Longleftarrow)$ By Eq.(\mref{eq:2.30}), we have
 \begin{eqnarray*}
 \Delta(a)=\alpha(1_{1})\otimes 1_{2}\psi\beta^{-1}(a)+\delta_{1}(a)=\omega\alpha^{-1}(a)1_{1}\otimes\beta(1_{2})+\delta_{2}(a),
 \quad\forall a\in A.
 \end{eqnarray*}
 So
 \begin{eqnarray}
 &&\Delta(a)\triangleleft b=\omega(a)\alpha(1_{1})\otimes \beta(1_{2})\psi(b)+\delta_{2}(a)\triangleleft b,\mlabel{eq:2.31}\\
 &&a\triangleright\Delta(b)=\omega(a)\alpha(1_{1})\otimes
 \beta(1_{2})\psi(b)+a\triangleright\delta_{1}(b).\mlabel{eq:2.32}
 \end{eqnarray}
 By Eqs.(\mref{eq:2.31}) and (\mref{eq:2.32}), then
 \begin{eqnarray*}
 \Delta(a)\triangleleft b+a\triangleright\delta_{1}(b)=a\triangleright\Delta(b)+\delta_{2}(a)\triangleleft b.
 \end{eqnarray*}
 While
 \begin{eqnarray*}
 \Delta(ab)&=&\omega\alpha^{-1}(ab)1_{1}\otimes\beta(1_{2})+\delta_{2}(ab)\\
 &\stackrel{(\mref{eq:2.8})}=&\omega\alpha^{-1}(ab)1_{1}\otimes\beta(1_{2})
 +a\triangleright\delta_{2}(b)+\delta_{2}(a)\triangleleft b
 \end{eqnarray*}
 and
 \begin{eqnarray*}
 a\triangleright\Delta(b)&=&\omega(a)(\omega\alpha^{-1}(b)1_{1})\otimes\beta^{2}(1_{2})
 +a\triangleright\delta_{2}(b)\\
 &\stackrel{(\mref{eq:1.3})}=&(\alpha^{-1}\omega(a)\omega\alpha^{-1}(b))\beta(1_{1})\otimes\beta^{2}(1_{2})
 +a\triangleright\delta_{2}(b)\\
 &\stackrel{(\mref{eq:1.2})(\mref{eq:2.14})}=&\omega\alpha^{-1}(ab)1_{1}\otimes\beta(1_{2})
 +a\triangleright\delta_{2}(b).
 \end{eqnarray*}
 Thus, $\Delta(ab)=\Delta(a)\triangleleft b+a\triangleright\delta_{1}(b)=a\triangleright\Delta(b)+\delta_{2}(a)\triangleleft b$, $\Delta$ is a covariant BiHom-derivation  with respect to $(\delta_{1},\delta_{2})$.

 By Eqs.(\mref{eq:2.28}) and (\mref{eq:2.29}), and the properties of BiHom-derivations and covariant BiHom-derivations, we have
 \begin{eqnarray*}
 &&\omega\alpha(1_{1})\otimes\alpha(\bar{1}_{1})
 \otimes\bar{1}_{2}(\psi\beta^{-1}(1_{2})\psi^{2}\beta^{-2}(a))
 +\omega\alpha(1_{1})\otimes\omega(1_{2})\psi\beta^{-1}(a)^{1}_{[1]}
 \otimes\beta(\psi\beta^{-1}(a)^{1}_{[2]})\\
 &&+\omega\alpha(1_{1})\otimes\alpha((1_{2})^{1}_{[1]})\otimes (1_{2})^{1}_{[2]}\psi^{2}\beta^{-1}(a)
 +\omega(a^{1}_{[1]})\otimes\alpha(1_{1})\otimes 1_{2}\psi\beta^{-1}(a^{1}_{[2]})\\
 &&+\omega(a^{1}_{[1]})\otimes(a^{1}_{[2]})^{1}_{[1]}\otimes(a^{1}_{[2]})^{1}_{[2]}\\
 &=&\alpha(\bar{1}_{1})\otimes \bar{1}_{2}\psi\beta^{-1}\alpha(1_{1})\otimes\psi(1_{2})\psi^{2}\beta^{-1}(a)
 +\alpha(1_{1})^{1}_{[1]}\otimes\alpha(1_{1})^{1}_{[2]}\otimes\psi(1_{2})\psi^{2}\beta^{-1}(a)\\
 &&+\alpha(1_{1})\otimes 1_{2}\psi\beta^{-1}(a^{1}_{[1]})\otimes\psi(a^{1}_{[2]})
 +(a^{1}_{[1]})^{1}_{[1]}\otimes(a^{1}_{[1]})^{1}_{[2]}\otimes\psi(a^{1}_{[2]}),
 \end{eqnarray*}
 which is exactly the condition of BiHom-coassociativity for $\Delta$ at any $a\in A$.
  \end{proof}

 \begin{rmk}\mlabel{rmk:2.9A} Let $\psi=\beta,\omega=\alpha$ in Theorem \mref{thm:2.9}, then we can recover \cite[Theorem 3.5]{MY}. If furthermore $\a=\b=\id$, then one gets Part (1) in \cite[Proposition 3.10]{Br}.
 \end{rmk}

 \begin{lem}\mlabel{lem:2.7} Let $(A,\mu,\alpha,\beta)$ be a BiHom-associative algebra such that $\alpha,\beta$ are bijective, $\psi,\omega:A\lr A$ be linear maps, $r, s \in A\otimes A$ be $\alpha,\beta,\psi,\omega$-invariant and moreover Eqs.(\mref{eq:2.12}) and (\mref{eq:2.14}) hold. Define the linear maps
 \begin{eqnarray}
 &\delta_{r}:A\lr A\otimes A,~~ \delta_{r}(a)=\alpha^{-1}(a)\triangleright r-r\triangleleft\beta^{-1}(a),&\mlabel{eq:2.17}\\
 &\delta_{s}:A\lr A\otimes A, ~~\delta_{s}(a)=\alpha^{-1}(a)\triangleright s-s\triangleleft\beta^{-1}(a),&\mlabel{eq:2.18}\\
 &\Delta:A\lr A\otimes A, ~~\Delta(a)=\alpha^{-1}(a)\triangleright r-s\triangleleft\beta^{-1}(a).&\mlabel{eq:2.19}
 \end{eqnarray}
 Then $\delta_{r}, \delta_{s}$ are BiHom-derivations and $\Delta$ is covariant BiHom-derivation with respect to BiHom-derivations $\delta_{r},\delta_{s}$.
 \end{lem}

 \begin{proof} For all $a, b\in A$, we have
 \begin{eqnarray*}
 a\triangleright\delta_{r}(b)+\delta_{r}(a)\triangleleft b 
 &\stackrel{(\mref{eq:2.3})(\mref{eq:2.4})}=&\omega(a)(\omega\alpha^{-1}(b)r^{1})\otimes
 \beta^{2}(r^{2})-\omega(a)\alpha(r^{1})\otimes \beta(r^{2})\psi(b)\\
 &&+\omega(a)\alpha(r^{1})\otimes\beta(r^{2})\psi(b)-\alpha^{2}(r^{1})\otimes (r^{2}\psi\beta^{-1}(a))\psi(b)\\ 
 &\stackrel{(\mref{eq:1.3})(\mref{eq:2.14})}=&\omega\alpha^{-1}(ab)r^{1}\otimes
 \beta(r^{2})-\alpha(r^{1})\otimes r^{2}\psi\beta^{-1}(ab)=\delta_{r}(ab).
 \end{eqnarray*}
 Then $\d_r$ is a BiHom-derivation. Similary, $\d_s$ is a BiHom-derivation. That $\Delta$ is covariant BiHom-derivation with respect to BiHom-derivations $\delta_{r},\delta_{s}$ can be verified as follows. For all $a,b\in A$, we have
 \begin{eqnarray*}
 \Delta(a)\triangleleft b+a\triangleright\delta_{r}(b) 
 &\stackrel{(\mref{eq:2.3})(\mref{eq:2.4})}=&\omega(a)\alpha (r^{1})\otimes\beta(r^{2})\psi(b)-\alpha^{2}(s^{1})\otimes (s^{2}\psi\beta^{-1}(a))\psi(b)\\
 &&+\omega(a)(\omega\alpha^{-1}(b)r^{1})\otimes\beta^{2}(r^{2})-\omega(a)\alpha(r^{1})\otimes \beta(r^{2})\psi(b)\\
 &=&\omega(a)(\omega\alpha^{-1}(b)r^{1})\otimes\beta^{2}(r^{2})-\alpha^{2}(s^{1})\otimes (s^{2}\psi\beta^{-1}(a))\psi(b)\\
 &\stackrel{(\mref{eq:1.3})}=&(\alpha^{-1}\omega(a)\omega\alpha^{-1}(b))\beta(r^{1})
 \otimes\beta^{2}(r^{2})\\
 &&-\alpha^{2}(s^{1})\otimes \alpha(s^{2})(\psi\beta^{-1}(a)\psi\beta^{-1}(b))\\
 &\stackrel{(\mref{eq:2.14})}=&\omega\alpha^{-1}(ab)\beta(r^{1})
 \otimes\beta^{2}(r^{2})-\alpha^{2}(s^{1})\otimes \alpha(s^{2})\psi\beta^{-1}(ab)=\Delta(ab)
 \end{eqnarray*}
 and likewise
 \begin{eqnarray*}
 a\triangleright\Delta(b)+\delta_{s}(a)\triangleleft b 
 =\Delta(ab),
 \end{eqnarray*}
 finishing the proof.
 \end{proof}

 \begin{pro}\mlabel{pro:2.11} Let $(A,\mu,\alpha,\beta)$ be a BiHom-assciative algebra such that $\alpha,\beta$ are bijective, and $\psi,\omega: A\lr A$ be linear maps, $r,s \in A\otimes A$ be $\alpha, \beta, \psi, \omega$-invariant, and moreover Eqs.(\mref{eq:2.12}) and (\mref{eq:2.14}) hold. Then $(A, \mu, \delta_{r}, \delta_{s}, \Delta, \alpha, \beta, \psi, \omega)$ is a covariant BiHom-bialgebra if and only if, for all $a\in A$,
 \begin{eqnarray}
 &\omega\alpha^{-1}(a)\triangleright(r_{13}r_{12}-r_{12}r_{23}+s_{23}r_{13})
 =(s_{13}r_{12}-s_{12}s_{23}+s_{23}s_{13})\triangleleft \psi\beta^{-1}(a).&\mlabel{eq:2.36}
 \end{eqnarray}
 \end{pro}

 \begin{proof}According to Lemma \mref{lem:2.7}, we can obtain that $\delta_{r}, \delta_{s}$ are BiHom-derivations and $\Delta$ is covariant BiHom-derivation with respect to $(\delta_{r}, \delta_{s})$. For all $a\in A$ and $r=\bar{r}, s=\bar{s}$, one can compute,
 \begin{eqnarray*}
 (\omega\otimes\Delta)\circ\Delta(a) 
 &\stackrel{(\mref{eq:2.19})}=&\omega^{2}\alpha^{-1}(a)\omega(r^{1})\otimes\alpha^{-1}\beta\omega(r^{2})\bar{r}^{1}\otimes\beta(\bar{r}^{2})\\
 &&-\omega^{2}\alpha^{-1}(a)\omega(r^{1})\otimes\alpha(s^{1})\otimes s^{2}\psi(r^{2})\\
 &&-\omega\alpha(s^{1})\otimes(\omega\alpha^{-1}(s^{2})\alpha^{-1}\beta^{-1}\psi\omega(a))
 r^{1}\otimes\beta(r^{2})\\
 &&+\omega\alpha(s^{1})\otimes\alpha(\bar{s}^{1})\otimes \bar{s}^{2}(\psi\beta^{-1}(s^{2})\psi^{2}\beta^{-2}(a))\\
 &\stackrel{(\mref{eq:1.3})}=&\omega^{2}\alpha^{-1}(a)r^{1}\otimes\alpha^{-1}\beta(r^{2})\bar{r}^{1}\otimes\beta(\bar{r}^{2})\\
 &&-\omega^{2}\alpha^{-1}(a)\omega(r^{1})\otimes\alpha(s^{1})\otimes s^{2}\psi(r^{2})\\
 &&-\alpha(s^{1})\otimes(\alpha^{-1}(s^{2})\alpha^{-1}\beta^{-1}\psi\omega(a))
 r^{1}\otimes\beta(r^{2})\\
 &&+\omega\alpha(s^{1})\otimes\alpha(\bar{s}^{1})\otimes (\alpha^{-1}(\bar{s}^{2})\psi\beta^{-1}(s^{2}))\psi^{2}\beta^{-1}(a),
 \end{eqnarray*}
 and likewise
 \begin{eqnarray*}
 (\Delta\otimes\psi)\circ\Delta(a) 
 &\stackrel{ }=&(\omega^{2}\alpha^{-2}(a)\omega\alpha^{-1}(r^{1}))\bar{r}^{1}
 \otimes\beta(\bar{r}^{2})\otimes\beta\psi(r^{2})\\
 &&-\alpha(s^{1})\otimes (\alpha^{-1}(s^{2})\alpha^{-1}\beta^{-1}\psi\omega(a))r^{1}\otimes\beta(r^{2})\\
 &&-\omega(s^{1})r^{1}\otimes\beta(r^{2})\otimes\psi(s^{2})\psi^{2}\beta^{-1}(a)\\
 &&+\alpha(\bar{s}^{1})\otimes \bar{s}^{2}\alpha\beta^{-1}\psi(s^{1})\otimes\psi(s^{2})\psi^{2}\beta^{-1}(a).
 \end{eqnarray*}
 Then $\Delta$ is BiHom-coassociative if and only if 
 \begin{eqnarray*}
 &&\omega^{2}\alpha^{-1}(a)(\omega(r^{1})\bar{r}^{1})
 \otimes\beta^{2}(\bar{r}^{2})\otimes\alpha\beta\psi(r^{2})
 -\omega^{2}\alpha^{-1}(a)\alpha(r^{1})\otimes\beta(r^{2})\beta(\bar{r}^{1})\otimes\beta^{2}(\bar{r}^{2})\\
 &&+\omega^{2}\alpha^{-1}(a)\beta\omega(r^{1})\otimes\alpha\beta(s^{1})\otimes \beta(s^{2})\beta\psi(r^{2})
 =\alpha\omega(s^{1})\alpha(r^{1})\otimes\alpha\beta(r^{2})\otimes\alpha\psi(s^{2})\psi^{2}\beta^{-1}(a)\\
 &&-\alpha^{2}(\bar{s}^{1})\otimes \alpha(\bar{s}^{2})\alpha(s^{1})\otimes\beta(s^{2})\psi^{2}\beta^{-1}(a)+\alpha\beta\omega(s^{1})\otimes\alpha^{2}(\bar{s}^{1})\otimes (\bar{s}^{2}\psi(s^{2}))\psi^{2}\beta^{-1}(a).
 \end{eqnarray*}
 and this is exactly Eq.(\mref{eq:2.36}).  These finish the proof.
 \end{proof}

 \begin{defi}\mlabel{de:2.13} Under the assumption of Proposition \mref{pro:2.11}, if $(r,s)$ is an abhYBp of weight $(0,0)$ in $(A,\mu,\alpha,\beta)$, then by Proposition \mref{pro:2.11}, $(A, \mu, \delta_{r}, \delta_{s}, \Delta, \alpha, \beta, \psi, \omega)$ is a covariant BiHom-bialgebra. In this case $(A, \mu, \delta_{r}, \delta_{s}, \Delta, \alpha, \beta, \psi, \omega)$ is called a {\bf quasitriangular covariant BiHom-bialgebra}.
 \end{defi}

 \begin{cor}\mlabel{cor:2.11a} Let $(A,\mu,\alpha,\beta)$ be a BiHom-assciative algebra such that $\alpha, \beta$ are bijective, and $\psi,\omega: A\lr A$ be linear maps, $r\in A\otimes A$ be $\alpha, \beta, \psi, \omega$-invariant, and moreover Eqs.(\mref{eq:2.12}) and (\mref{eq:2.14}) hold. Then $(A, \mu, \delta_{r}=\delta_{s}=\Delta, \alpha, \beta, \psi, \omega)$ is an infinitesimal BiHom-bialgebra if and only if, for all $a\in A$,
 \begin{eqnarray}
 &\omega\alpha^{-1}(a)\triangleright(r_{13}r_{12}-r_{12}r_{23}+r_{23}r_{13})
 =(r_{13}r_{12}-r_{12}r_{23}+r_{23}r_{13})\triangleleft \psi\beta^{-1}(a).&\mlabel{eq:2.36a}
 \end{eqnarray}
 \end{cor}

 \begin{proof} Let $r=s$ in Proposition \mref{pro:2.11}.
 \end{proof}

 \begin{defi}\mlabel{de:2.14} Under the assumption of Corollary \mref{cor:2.11a}, if $r$ is a solution to abhYBe of weight $0$ in $(A,\mu,\alpha,\beta)$, then by Corollary \mref{cor:2.11a}, we know $(A, \mu, \delta_{r}=\delta_{s}=\Delta, \alpha, \beta, \psi, \omega)$ is a covariant BiHom-bialgebra. In this case $(A, \mu, \Delta, \alpha, \beta, \psi, \omega)$ is the BiHom-type of quasitriangular infinitesimal bialgebra introduced in \cite{Ag99}, called {\bf quasitriangular infinitesimal BiHom-bialgebra}.
 \end{defi}

 \begin{rmk}\mlabel{rmk:2.14aa} (1) We note here the definition of quasitriangular infinitesimal BiHom-bialgebra is different from the one introduced by Liu, Makhlouf, Menini, Panaite in \cite[Definition 5.7]{LMMP3}. And if $\psi=\b, \om=\a$ and by $r$ is $\psi, \om$-invariant, then Definition \mref{de:2.14} can cover \cite[Definition 5.7]{LMMP3}.

 (2) Definition \mref{de:2.14} is really the BiHom-type of quasitriangular infinitesimal bialgebra since it contains four structure maps $\a, \b, \psi, \om$.
 \end{rmk}

 \begin{thm}\mlabel{thm:2.15} Let $(A,\mu,\alpha,\beta)$ be a BiHom-associative algebra,  $\psi, \om: A\lr A$ two linear maps such that any two of $\a, \b, \psi, \om$ commute,  $r,s\in A\otimes A$ be $\a, \b, \psi, \om$-invariant, and $\delta_{r}, \delta_{s}, \Delta$ be defined by Eqs.(\mref{eq:2.17})-(\mref{eq:2.19}). Then $(A, \mu, \delta_{r}, \delta_{s}, \Delta, \alpha, \beta, \psi, \omega)$ is a quasitriangular covariant BiHom-bialgebra if and only if
 \begin{eqnarray}
 &(\omega\otimes \Delta)(r)=r_{13}r_{12},&\mlabel{eq:2.39}\\
 &(\Delta\otimes\psi)(s)=-s_{23}s_{13}.&\mlabel{eq:2.40}
 \end{eqnarray}
 \end{thm}

 \begin{proof} By the definition of $\mu$ in Eq.(\mref{eq:2.19}), one easily checks that
 \begin{eqnarray*}
 (\omega\otimes \Delta)(r) 
 &\stackrel{(\mref{eq:2.19})}=&\omega(r^{1})\otimes\omega\alpha^{-1}(r^{2})\bar{r}^{1}
 \otimes\beta(\bar{r}^{2})-\omega(r^{1})\otimes\alpha(s^{1})\otimes s^{2}\psi\beta^{-1}(r^{2})\\
 &=&\alpha(r^{1})\otimes r^{2}\bar{r}^{1}
 \otimes\beta(\bar{r}^{2})-\beta\omega(r^{1})\otimes\alpha(s^{1})\otimes s^{2}\psi(r^{2})\\
 &\stackrel{(\mref{eq:4.53})}=&r_{12}r_{23}-s_{23}r_{13}
 \end{eqnarray*}
 and
 \begin{eqnarray*}
 (\Delta\otimes\psi)(s) 
 &\stackrel{(\mref{eq:2.19})}=&\omega\alpha^{-1}(s^{1})r^{1}\otimes\beta(r^{2})
 \otimes\psi(s^{2})-\alpha(\bar{s}^{1})\otimes \bar{s}^{2}\psi\beta^{-1}(s^{1})\otimes\psi(s^{2})\\
 &=&\omega(s^{1})r^{1}\otimes\beta(r^{2})
 \otimes\alpha\psi(s^{2})-\alpha(\bar{s}^{1})\otimes \bar{s}^{2}s^{1}\otimes\beta(s^{2})\\
 &\stackrel{(\mref{eq:4.54})}=&s_{13}r_{12}-s_{12}s_{23}.
 \end{eqnarray*}
 Based on the above calculations, Eq.(\mref{eq:2.39}) is equivalent to Eq.(\mref{eq:4.53}) and Eq.(\mref{eq:2.40}) is equivalent to Eq.(\mref{eq:4.54}).
 \end{proof}

 \begin{cor}\mlabel{cor:2.16} Let $(A,\mu,\alpha,\beta)$ be a BiHom-associative algebra,  $\psi, \om: A\lr A$ two linear maps such that any two of $\a, \b, \psi, \om$ commute,   $r \in A\otimes A$ be $\a, \b, \psi, \om$-invariant, and $\Delta$ be defined by Eq.(\mref{eq:2.17}). Then $(A,\mu,\Delta,\alpha,\beta,\psi,\omega)$ is a quasitriangular infinitesimal BiHom-bialgebra if and only if Eq.(\mref{eq:2.39}) or
 \begin{eqnarray}
 &(\Delta\otimes\psi)(r)=-r_{23}r_{13}&\mlabel{eq:2.41}
 \end{eqnarray}
 holds.
 \end{cor}

 \begin{proof}Let $r=s$ in Theorem \mref{thm:2.15}.
 \end{proof}

 \begin{pro}\mlabel{pro:2.17} Under the assumption of Proposition \mref{pro:2.11}, $(A,\mu,\delta_{r},\delta_{s},\Delta,\alpha,\beta,\psi,\omega)$ is a unitary quasitriangular covariant BiHom-bialgebra if and only if
 \begin{eqnarray}
 &\omega(r^{1})\otimes 1\otimes\psi(r^{2})=r_{13}r_{12}-r_{12}r_{23}+r_{23}r_{13},&\mlabel{eq:2.42}
 \end{eqnarray}
 or equivalently, Eq.(\mref{eq:2.39}) and
 \begin{eqnarray}
 &(\Delta\otimes\psi)(r)=-r_{23}r_{13}+1\otimes r^{1}\otimes r^{2}+\omega(r^{1})\otimes 1\otimes\psi(r^{2})&\mlabel{eq:2.43}
 \end{eqnarray}
 hold.
 \end{pro}

 \begin{proof} 
 By Eqs.(\mref{eq:4.53}), (\mref{eq:4.54}) for the case of weight $(0, 0)$ and $s=r-1\otimes 1$, we have Eq.(\mref{eq:2.42}). In addition, we complete the proof by replacing $s$ by $r-1\otimes1$ in Eqs.(\mref{eq:2.39}), (\mref{eq:2.40}).
 \end{proof}

 \begin{rmk} Eq.(\mref{eq:2.42}) is exactly the abhYBe of weight $-1$. Thus unitary quasitriangular covariant BiHom-bialgebra can provide the solution to the nonhomogeneous abhYBe.
 \end{rmk}

 \section{Rota-Baxter paired BiHom-module of weight $\lambda$}\label{se:rbpm} Rota-Baxter paired modules was introduced by Zheng, Guo and Zhang in \cite{ZGZ} in order to study Rota-Baxter module. In this section, we present two approaches to construct (BiHom-)pre-Lie modules from Rota-Baxter (BiHom-)paired modules.

\subsection{First approach}
 \begin{lem}\mlabel{lem:3.4} (\cite{MYZZ}) If $(A, R, \alpha, \beta)$ is a Rota-Baxter BiHom-associative algebra of weight $\lambda$, then $(A, R, R+\lambda \id, \alpha, \beta)$ and $(A, R+\lambda \id, R, \alpha, \beta)$ are Rota-Baxter BiHom-systems.
 \end{lem}

 Let us recall from \cite{LMMP1} the following definition: A {\bf BiHom-dendriform algebra} is a vector space $A$ together with two binary operations $\prec, \succ: A \otimes A \rightarrow A$ and two commuting linear maps $\alpha, \beta: A \rightarrow A$ such that for all $a,b,c \in A$,
 \begin{eqnarray*}
 &\alpha(a\prec b)=\alpha(a)\prec\alpha(b),\quad \alpha(a\succ b)=\alpha(a)\succ\alpha(b),&\\%\mlabel{eq:4.21}
 &\beta(a\prec b)=\beta(a)\prec\beta(b),\quad \beta(a\succ b)=\beta(a)\succ\beta(b),&\\%\mlabel{eq:4.22}
 &(a\prec b)\prec\beta(c)=\alpha(a)\prec(b\prec c)+\alpha(a)\prec (b\succ c),&\\%\mlabel{eq:4.23}
 &(a\succ b)\prec\beta(c)=\alpha(a)\succ(b\prec c),&\\%\mlabel{eq:4.24}
 &\alpha(a)\succ(b\succ c)=(a\prec b)\succ\beta(c)+ (a\succ b)\succ\beta(c).&%\mlabel{eq:4.25}
 \end{eqnarray*}
 For simplicity, we denote it by $(A,\prec, \succ, \alpha, \beta)$.

 \begin{lem}\mlabel{lem:3.8} (\cite{MYZZ}) Let $(A, R, S, \alpha, \beta)$ be a Rota-Baxter BiHom-system and linear maps $\prec,\succ: A\otimes A\longrightarrow A$ be defined by
 \begin{eqnarray*}
 &a\prec b=aS(b),\quad a\succ b=R(a)b&%\mlabel{eq:3.16}
 \end{eqnarray*}
 for all $a,b\in A$. Then $(A, \prec, \succ, \alpha, \beta)$ is a BiHom-dendriform algebra.
 \end{lem}

 \begin{cor}\mlabel{cor:3.15} Let $(A,R,\alpha,\beta)$ be a Rota-Baxter BiHom-associative algebra of weight $\lambda$, and let linear maps $\prec,\succ: A\otimes A\lr A$ be defined by
 \begin{eqnarray}
 &a\prec b=aR(b),\quad a\succ b=R(a)b+\lambda ab& \mlabel{eq:3.28}
 \end{eqnarray}
 for all $a,b\in A$. Then $(A,\prec,\succ,\alpha,\beta)$ is a BiHom-dendriform algebra.
 \end{cor}

 \begin{proof} By Lemma \mref{lem:3.4}, we know that $(A, R+\lambda \id, R, \alpha, \beta)$ is Rota-Baxter BiHom-system, then the proof is finished by Lemma \mref{lem:3.8}.
 \end{proof}

 \begin{defi} (\cite{LMMP6}) A {\bf BiHom-pre-Lie algebra} is a vector space $A$ together with a binary operation $\ast: A \otimes A \rightarrow A$ and two commuting linear maps $\alpha,\beta: A \rightarrow A$ such that
 \begin{eqnarray*}
 &\alpha(a\ast b)=\alpha(a)\ast \alpha(b),\quad \beta(a\ast b)=\beta(a)\ast \beta(b),& \\
 &\alpha\beta(a)\ast (\alpha(b)\ast c)-(\beta(a)\ast \alpha(b))\ast \beta(c)
 =\alpha\beta(b)\ast (\alpha(a)\ast c)
 -(\beta(b)\ast \alpha(a))\ast \beta(c)&%\mlabel{eq:4.18}
 \end{eqnarray*}
 for all $a,b,c \in A$.
 \end{defi}

 \begin{pro}\mlabel{pro:3.16} Let $(A,R,\alpha,\beta)$ be a Rota-Baxter BiHom-associative algebra of weight $\lambda$ such that $\alpha,\beta$ are bijective. Define the operation $\star$ by
 \begin{eqnarray}
 &a\star b=R(a)b+\lambda ab-\alpha^{-1}\beta(b)R(\alpha\beta^{-1}(a))&\mlabel{eq:3.29}
 \end{eqnarray}
 for all $a,b\in A$. Then $(A,\star,\alpha,\beta)$ is a left BiHom-pre-Lie algebra.
 \end{pro}

 \begin{proof} By Corollary \mref{cor:3.15} and \cite[Proposition 2.6]{LMMP6}, the proof is finished.
 \end{proof}

 \begin{defi}\mlabel{de:3.13}(\cite{MYZZ}) Fix a $\lambda\in K$. Let $(A, \mu, \alpha_{A}, \beta_{A})$ be a BiHom-associative algebra and $(M, \alpha_{M}, \beta_{M})$ a left $(A, \mu, \alpha_{A}, \beta_{A})$-module with action $\triangleright$. A pair $(R, T)$ of linear maps $R: A\lr A$ and $T:M\lr M$ is called a {\bf Rota-Baxter BiHom-paired operator of weight $\lambda$ on $(M, \alpha_{M}, \beta_{M})$} if
 \begin{eqnarray}
 &\alpha_{M}\circ T=T\circ\alpha_{M},~~ \beta_{M}\circ T=T\circ\beta_{M},~~ \alpha_{A}\circ R=R\circ\alpha_{A},~~ \beta_{A}\circ R=R\circ\beta_{A},&\mlabel{eq:3.24}\\
 &R(a)\triangleright T(m)=T(R(a)\triangleright m)+T(a\triangleright T(m))+\lambda T(a\triangleright m)&\mlabel{eq:3.25}
 \end{eqnarray}
 for all $a\in A$ and $m\in M$.

 We also call the $7$-tuple $(M, R, T, \alpha_{M}, \beta_{M}, \alpha_{A}, \beta_{A})$ a {\bf Rota-Baxter paired (left) $(A, \mu, \alpha_{A}, \beta_{A})$-module of weight $\lambda$}.
 \end{defi}

 \begin{rmk} If the structure maps $\alpha_{A}=\beta_{A}=\id_A$ and $\alpha_{M}=\beta_{M}=\id_M$, then we can get the usual notion of Rota-Baxter paired module introduced in \cite{ZGZ}.
 \end{rmk}

 \begin{defi}\mlabel{de:3.14} Let $(A,\mu,\alpha_{A},\beta_{A})$ be a (left) BiHom-pre-Lie algebra. A left {\bf $(A,\mu,\alpha_{A},\beta_{A})$-pre-Lie module} is a quadruple $(M,\triangleright,\alpha_{M},\beta_{M})$, where $M$ is a linear space and $\triangleright: A\otimes M\lr M$, $\alpha_{M},\beta_{M}: M\lr M$ are linear maps satisfying
 \begin{eqnarray}
 &\alpha_{M}\circ \beta_{M}=\beta_{M}\circ \alpha_{M},~~\alpha_{M}(a\triangleright m)=\alpha_{A}(a)\triangleright\alpha_{M}(m),~~\beta_{M}(a\triangleright m)=\beta_{A}(a)\triangleright\beta_{M}(m),&\mlabel{eq:3.26}\\
 &\alpha_{A}\beta_{A}(a)\triangleright(\alpha_{A}(b)\triangleright m)-(\beta_{A}(a)\cdot\alpha_{A}(b))\triangleright\beta_{M}(m)&\nonumber \\
 &\qquad\qquad\qquad\qquad\qquad\qquad\qquad=\alpha_{A}\beta_{A}(b)\triangleright(\alpha_{A}(a)
 \triangleright m)-(\beta_{A}(b)\cdot\alpha_{A}(a))\triangleright\beta_{M}(m)  &\mlabel{eq:3.27}
 \end{eqnarray}
 for all $a,b \in A$ and $m\in M$.

 \end{defi}

 \begin{thm}\mlabel{thm:3.17} Let $(A,R,\alpha_{A},\beta_{A})$ be a Rota-Baxter BiHom-associative algebra of weight $\lambda$ and $(M, R$, $T, \alpha_{M}, \beta_{M}, \alpha_{A}, \beta_{A})$ a Rota-Baxter paired $(A, \mu, \alpha_{A}, \beta_{A})$-module of weight $\lambda$. Define a map $\triangleright_{\star}: A\otimes M\lr M$ by
 \begin{eqnarray}
 &a\triangleright_{\star} m=R(a)\triangleright m+a\triangleright T(m)+\lambda a\triangleright m &\mlabel{eq:3.30}
 \end{eqnarray}
 for all $a\in A$ and $m\in M$. Then $(M, \triangleright_{\star}, \alpha_{M}, \beta_{M})$ is a left $(A, \star, \alpha_{A}, \beta_{A})$-pre-Lie module, where $\star$ is defined by Eq.(\mref{eq:3.29}).
 \end{thm}

 \begin{proof} By Proposition \mref{pro:3.16}, we know that $(A, \star, \alpha_{A}, \beta_{A})$ is a BiHom-pre-Lie algebra. Eq.(\mref{eq:3.26}) can be verified by Eq.(\mref{eq:3.24}).  So, next we only need to prove that Eq.(\mref{eq:3.27}) holds. For all $a, b\in A$ and $m\in M$, we have
 \begin{eqnarray*}
 &&\alpha_{A}\beta_{A}(a)\triangleright_{\star}(\alpha_{A}(b)\triangleright_{\star} m)-(\beta_{A}(a)\star\alpha_{A}(b))\triangleright_{\star}\beta_{M}(m)\\
 &&\qquad\stackrel{(\mref{eq:3.30})(\mref{eq:3.29})}= 
 \alpha_{A}\beta_{A}R(a)\triangleright(\alpha_{A}R(b)\triangleright m)+\alpha_{A}\beta_{A}(a)\triangleright T(\alpha_{A}R(b)\triangleright m)\\
 &&\qquad\qquad\quad+\lambda\alpha_{A}\beta_{A}(a)\triangleright(\alpha_{A}R(b)\triangleright m)+\alpha_{A}\beta_{A}R(a)\triangleright(\alpha_{A}(b)\triangleright T(m))\\
 &&\qquad\qquad\quad+\alpha_{A}\beta_{A}(a)\triangleright T(\alpha_{A}(b)\triangleright T(m))+\lambda\alpha_{A}\beta_{A}(a)\triangleright(\alpha_{A}(b)\triangleright T(m))\\
 &&\qquad\qquad\quad+\lambda\alpha_{A}\beta_{A}R(a)\triangleright(\alpha_{A}R(b)\triangleright m)+\lambda\alpha_{A}\beta_{A}(a)\triangleright T(\alpha_{A}(b)\triangleright m)\\
 &&\qquad\qquad\quad+\lambda^{2}\alpha_{A}\beta_{A}(a)\triangleright(\alpha_{A}(b)\triangleright m)-R(R(\beta_{A}(a))\alpha_{A}(b))\triangleright\beta_{M}(m)\\
 &&\qquad\qquad\quad-R(\beta_{A}(a))\alpha_{A}(b)\triangleright T(\beta_{M}(m))
 -\lambda R(\beta_{A}(a))\alpha_{A}(b)\triangleright\beta_{M}(m)\\
 &&\qquad\qquad\quad-\lambda R(\beta_{A}(a)\alpha_{A}(b))\triangleright\beta_{M}(m)
 -\lambda\beta_{A}(a)\alpha_{A}(b)\triangleright T(\beta_{M}(m))\\
 &&\qquad\qquad\quad-\lambda^{2} \beta_{A}(a)\alpha_{A}(b)\triangleright\beta_{M}(m)
 +R(\beta_{A}(b)R(\alpha_{A}(a)))\triangleright\beta_{M}(m)\\
 &&\qquad\qquad\quad+\beta_{A}(b)R(\alpha_{A}(a))\triangleright T(\beta_{M}(m))
 +\lambda\beta_{A}(b)\alpha_{A}R(a)\triangleright\beta_{M}(m)\\
 &&\qquad\quad\stackrel{(\mref{eq:1.15})}=(R\beta_{A}(a)R\alpha_{A}(b))\triangleright\beta_{M}(m)
 +\alpha_{A}\beta_{A}(a)\triangleright T(\alpha_{A}R(b)\triangleright m)\\
 &&\qquad\qquad\quad+\lambda\alpha_{A}\beta_{A}(a)\triangleright(R\alpha_{A}(b)\triangleright m)+(R\beta_{A}(a)\alpha_{A}(b))\triangleright \beta_{M}T(m)\\
 &&\qquad\qquad\quad+\alpha_{A}\beta_{A}(a)\triangleright T(\alpha_{A}(b)\triangleright T(m))+\lambda\beta_{A}(a)\alpha_{A}(b)\triangleright\beta_{M}T(m)\\
 &&\qquad\qquad\quad+\lambda\beta_{A}R(a)\alpha_{A}(b)\triangleright \beta_{M}(m)
 +\lambda\alpha_{A}\beta_{A}(a)\triangleright T(\alpha_{A}(b)\triangleright m)\\
 &&\qquad\qquad\quad+\lambda^{2}\beta_{A}(a)\alpha_{A}(b)\triangleright\beta_{M}(m)
 -R(R\beta_{A}(a)\alpha_{A}(b))\triangleright\beta_{M}(m)\\
 &&\qquad\qquad\quad-R(\beta_{A}(a))\alpha_{A}(b)\triangleright T(\beta_{M}(m))
 -\lambda R(\beta_{A}(a))\alpha_{A}(b)\triangleright\beta_{M}(m)\\
 &&\qquad\qquad\quad-\lambda R(\beta_{A}(a)\alpha_{A}(b))\triangleright\beta_{M}(m)
 -\lambda\beta_{A}(a)\alpha_{A}(b)\triangleright T(\beta_{M}(m))\\
 &&\qquad\qquad\quad-\lambda^{2} \beta_{A}(a)\alpha_{A}(b)\triangleright\beta_{M}(m)
 +R(\beta_{A}(b)R\alpha_{A}(a))\triangleright\beta_{M}(m)\\
 &&\qquad\qquad\quad+\beta_{A}(b)R\alpha_{A}(a)\triangleright T(\beta_{M}(m))
 +\lambda\beta_{A}(b)\alpha_{A}R(a)\triangleright\beta_{M}(m)\\
 &&\qquad\quad\stackrel{ }=R(\beta_{A}(a)R\alpha_{A}(b))\triangleright\beta_{M}(m)
 +R(R\beta_{A}(a)\alpha_{A}(b))\triangleright\beta_{M}(m)\\
 &&\qquad\qquad\quad+\lambda R(\beta_{A}(a)\alpha_{A}(b))\triangleright\beta_{M}(m)
 +\alpha_{A}\beta_{A}(a)\triangleright T(\alpha_{A}R(b)\triangleright m)\\
 &&\qquad\qquad\quad+\lambda\alpha_{A}\beta_{A}(a)\triangleright(R\alpha_{A}(b)\triangleright m)
 +\alpha_{A}\beta_{A}(a)\triangleright T(\alpha_{A}(b)\triangleright T(m))\\
 &&\qquad\qquad\quad+\lambda\alpha_{A}\beta_{A}(a)\triangleright T(\alpha_{A}(b)\triangleright m)
 -R(R\beta_{A}(a)\alpha_{A}(b))\triangleright\beta_{M}(m)\\
 &&\qquad\qquad\quad-\lambda R(\beta_{A}(a)\alpha_{A}(b))\triangleright\beta_{M}(m)
 +R(\beta_{A}(b)R\alpha_{A}(a))\triangleright\beta_{M}(m)\\
 &&\qquad\qquad\quad+\beta_{A}(b)R\alpha_{A}(a)\triangleright T(\beta_{M}(m))
 +\lambda\beta_{A}(b)\alpha_{A}R(a)\triangleright\beta_{M}(m)\\
 &&\qquad\quad\stackrel{(\mref{eq:1.15})}=R(\beta_{A}(a)R\alpha_{A}(b))\triangleright\beta_{M}(m)
 +\alpha_{A}\beta_{A}(a)\triangleright T(\alpha_{A}R(b)\triangleright m)\\
 &&\qquad\qquad\quad+\lambda\alpha_{A}\beta_{A}(a)\triangleright(R\alpha_{A}(b)\triangleright m)+\alpha_{A}\beta_{A}(a)\triangleright T(\alpha_{A}(b)\triangleright T(m))\\
 &&\qquad\qquad\quad+\lambda\alpha_{A}\beta_{A}(a)\triangleright T(\alpha_{A}(b)\triangleright m)
 +R(\beta_{A}(b)R\alpha_{A}(a))\triangleright\beta_{M}(m)\\
 &&\qquad\qquad\quad+\alpha_{A}\beta_{A}(b)\triangleright (R\alpha_{A}(a)\triangleright T(m))
 +\lambda\alpha_{A}\beta_{A}(b)\triangleright(R\alpha_{A}(a)\triangleright m)\\
 &&\qquad\quad\stackrel{(\mref{eq:3.25})}=R(\beta_{A}(a)R\alpha_{A}(b))\triangleright\beta_{M}(m)
 +\alpha_{A}\beta_{A}(a)\triangleright T(\alpha_{A}R(b)\triangleright m)\\
 &&\qquad\qquad\quad+\lambda\alpha_{A}\beta_{A}(a)\triangleright(R\alpha_{A}(b)\triangleright m)+\alpha_{A}\beta_{A}(a)\triangleright T(\alpha_{A}(b)\triangleright T(m))\\
 &&\qquad\qquad\quad+\lambda\alpha_{A}\beta_{A}(a)\triangleright T(\alpha_{A}(b)\triangleright m)
 +R(\beta_{A}(b)R\alpha_{A}(a))\triangleright\beta_{M}(m)\\
 &&\qquad\qquad\quad+\alpha_{A}\beta_{A}(b)\triangleright T(R\alpha_{A}(a)\triangleright m)
 +\alpha_{A}\beta_{A}(b)\triangleright T(\alpha_{A}(a)\triangleright T(m))\\
 &&\qquad\qquad\quad+\lambda\alpha_{A}\beta_{A}(b)\triangleright T(\alpha_{A}(a)\triangleright m)+\lambda\alpha_{A}\beta_{A}(b)\triangleright(R\alpha_{A}(a)\triangleright m).
 \end{eqnarray*}
 We find that if exchanging $a$ and $b$ in the last item of the above equation, then the last item remains unchanged. Therefore we can finish the proof.
 \end{proof}

 If $\alpha_{A}=\beta_{A}=\id_A$ and $\alpha_{M}=\beta_{M}= \id_M$ in Theorem \mref{thm:3.17}, then we have

 \begin{cor}\mlabel{cor:3.18} Let $(A,R)$ be a Rota-Baxter algebra of weight $\lambda$  and $(M,R,T)$ a Rota-Baxter paired $A$-module of weight $\lambda$. Define a map $\triangleright_{\star}: A\otimes M\lr M$ by Eq.(\mref{eq:3.30}). 
 Then $(M,\triangleright_{\star})$ is a left $(A,\star)$-pre-Lie module, where $\star$ is defined by Eq.(\mref{eq:3.29}).
 \end{cor}

\subsection{Second approach}
 The following are parallel conclusions of the above results, so we omit the proof.

 \begin{lem}\mlabel{lem:3.19} Let $(A,R,\alpha,\beta)$ be a Rota-Baxter BiHom-associative algebra of weight $\lambda$ and linear maps $\prec,\succ: A\otimes A\lr A$ be defined by
 \begin{eqnarray*}
 &a\prec b=aR(b)+\lambda ab,\quad a\succ b=R(a)b& %\mlabel{eq:3.31}
 \end{eqnarray*}
 for all $a,b\in A$. Then $(A, \prec, \succ, \alpha, \beta)$ is a BiHom-dendriform algebra.
 \end{lem}

 \begin{lem}\mlabel{lem:3.20} Let $(A,R,\alpha,\beta)$ be a Rota-Baxter BiHom-associative algebra of weight $\lambda$ such that $\alpha,\beta$ are bijective. Define the operation $\natural$ by
 \begin{eqnarray*}
 &a\natural b=R(a)b-\alpha^{-1}\beta(b)R(\alpha\beta^{-1}(a))
 -\lambda\alpha^{-1}\beta(b)\alpha\beta^{-1}(a)& %\mlabel{eq:3.32}
 \end{eqnarray*}
 for all $a,b\in A$. Then $(A,\natural,\alpha,\beta)$ is a left BiHom-pre-Lie algebra.
 \end{lem}

 \begin{thm}\mlabel{thm:3.21} Let $(A,R,\alpha_{A},\beta_{A})$ be a Rota-Baxter BiHom-associative algebra of weight $\lambda$ and $(M, R$, $T, \alpha_{M}, \beta_{M}, \alpha_{A}, \beta_{A})$ a Rota-Baxter paired $(A, \mu, \alpha_{A}, \beta_{A})$-module of weight $\lambda$. Define a map $\triangleright_{\star}: A\otimes M\lr M$ by Eq.(\mref{eq:3.30}). Then $(M, \triangleright_{\star}, \alpha_{M}, \beta_{M})$ is a left $(A, \natural, \alpha_{A}, \beta_{A})$-pre-Lie module.
 \end{thm}

 If $\alpha_{A}=\beta_{A}=\id_A$ and $\alpha_{M}=\beta_{M}= \id_M$ in Theorem \mref{thm:3.21}, then we have

 \begin{cor}\mlabel{cor:3.22} Let $(A,R)$ be a Rota-Baxter algebra of weight $\lambda$  and $(M,R,T)$ a Rota-Baxter paired $A$-module of weight $\lambda$. Define a map $\triangleright_{\star}: A\otimes M\lr M$ by Eq.(\mref{eq:3.30}). Then $(M, \triangleright_{\star})$ is a left $(A, \natural)$-pre-Lie module.
 \end{cor} 

\section*{Acknowledgment} This work is supported by Natural Science Foundation of Henan Province (No. 212300410365).

 \end{document}